\documentclass[12pt]{amsart}
\usepackage[T1]{fontenc}
\usepackage[utf8]{inputenc}
\usepackage{amssymb,amsmath,hyperref}
\usepackage{mathtools}
\usepackage{enumerate, graphicx}
\usepackage[nobysame]{amsrefs}
\usepackage{amsthm}
\usepackage{url}
\usepackage{wrapfig}
\usepackage[a4paper, margin=25mm]{geometry}

\usepackage{xcolor}
\usepackage{caption}

\newtheorem{theorem}{Theorem}[section]

\newtheorem{lemma}[theorem]{Lemma}
\newtheorem{proposition}[theorem]{Proposition}

\newtheorem{conj}[theorem]{Conjecture}

\theoremstyle{definition}
\newtheorem{defin}[theorem]{Definition}

\theoremstyle{remark}
\newtheorem{remark}[theorem]{Remark}
\newtheorem*{remark*}{Remark}

\author{Kinga Nagy}
\address{Department of Geometry, Bolyai Institute, 
University of Szeged, Aradi v\'ertan\'uk tere 1, H-6720 Szeged, Hungary}
\email{kinga1204@live.com}

\author{Viktor V\'{\i}gh}
\address{Department of Geometry, Bolyai Institute, 
University of Szeged, Aradi v\'ertan\'uk tere 1, H-6720 Szeged, Hungary}
\email{vigvik@math.u-szeged.hu}

\title{ Monohedral Tilings of a Convex Disc with a Smooth Boundary}

\date{\today}

\DeclareMathOperator{\conv}{conv}
\DeclareMathOperator{\relint}{relint}
\DeclareMathOperator{\inter}{int}
\DeclareMathOperator{\tc}{tc}

\DeclarePairedDelimiter{\length}{\big\vert}{\big\vert}


\begin{document}

\begin{abstract}
In this paper we give a complete description about normal monohedral tilings of a convex disc with smooth boundary where we have at most three topological discs as tiles.
This result is a far-reaching generalization of the results of Kurusa, Lángi and Vígh \cite{KLV2020}. Some further partial results are proved for non-normal tilings.
\end{abstract}

\maketitle

\section{Introduction}

The compact image of the closed circular disc $\mathcal B^2$ under a continuous function is called a {\it topological disc}. The boundary of
a topological disc is called a {\it Jordan curve}. A finite set of topological discs is called a {\it tiling} of $K$ if the union of the discs is $K$, and their interiors are pairwise disjoint. A tiling is {\it monohedral} if all the tiles are congruent.

In this paper, we deal with the monohedral tilings of convex topological discs. The problem originates from a question attributed to Stein in \cite{croftfalconerguy}: is there a monohedral tiling of $\mathcal B^2$ in which the origin $O$ is in the interior of a tile? (See also \cite{MathOverflow}.)

The simplest monohedral tilings of $\mathcal B^2$ are the so-called {\it rotationally generated} tilings, which can be obtained by rotating a simple curve connecting $O$ and a boundary point of $K$ around $O$ by an angle of $k\cdot 2\pi /n$ for $k=0,1,\ldots,n-1$, see the first tiling of Figure~\ref{CircularDisc}.

\begin{figure}[ht]
{\centering
\includegraphics[width=\textwidth]{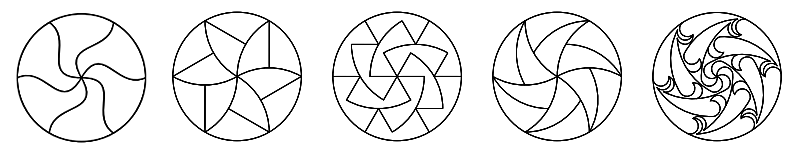}
\caption{\label{CircularDisc} Monohedral tilings of the closed circular disc $\mathcal B^2$. \cite{HaddleyWorsley2016}}
}
\end{figure}

In any rotationally generated tiling of $\mathcal B^2$, the center is contained in the boundary of every tile. This motivates the much studied problem of finding a tiling where one of the tiles does not contain (even on its boundary) the center of the disc. 
Some tilings of this type are well-known and appear in various places \cites{MASS, KoMaL}. Figure \ref{CircularDisc}. presents some constructions with this property.
Haddley and Worsley systematically analyse and construct tilings of this kind in \cite{HaddleyWorsley2016}.
We also mention Goncharov \cite{Goncharov}, which deals with similar problems in normed vector spaces, and Richter \cite{R2008}, which proves an assertion similar to Lemma~\ref{intersection}.
The recently published paper \cite{KLV2020} precisely describes the monohedral tilings of $\mathcal B^2$ in the case of at most three tiles: every tiling with $k$ tiles has a $k$-fold rotational symmetry $(k\leq 3)$. Very recently similar results were obtained by Basit and Lángi \cite{BL21}, they study monohedral tilings of regular $n$-gons.

A similar problem is to consider the monohedral tilings of rectangles, or in a more special case, squares. 
In \cite{Maltby1994}, Maltby proves that the tiles are rectangles in every monohedral tiling of a rectangle with three tiles. More recently, it was shown in \cite{YZZ2016} by Yuan, Zamfirescu and Zamfirescu that there is a unique monohedral tiling of the square with five convex tiles, see Figure~\ref{Nondiff}.

\begin{figure}[ht]
{\centering
\includegraphics[width=\textwidth]{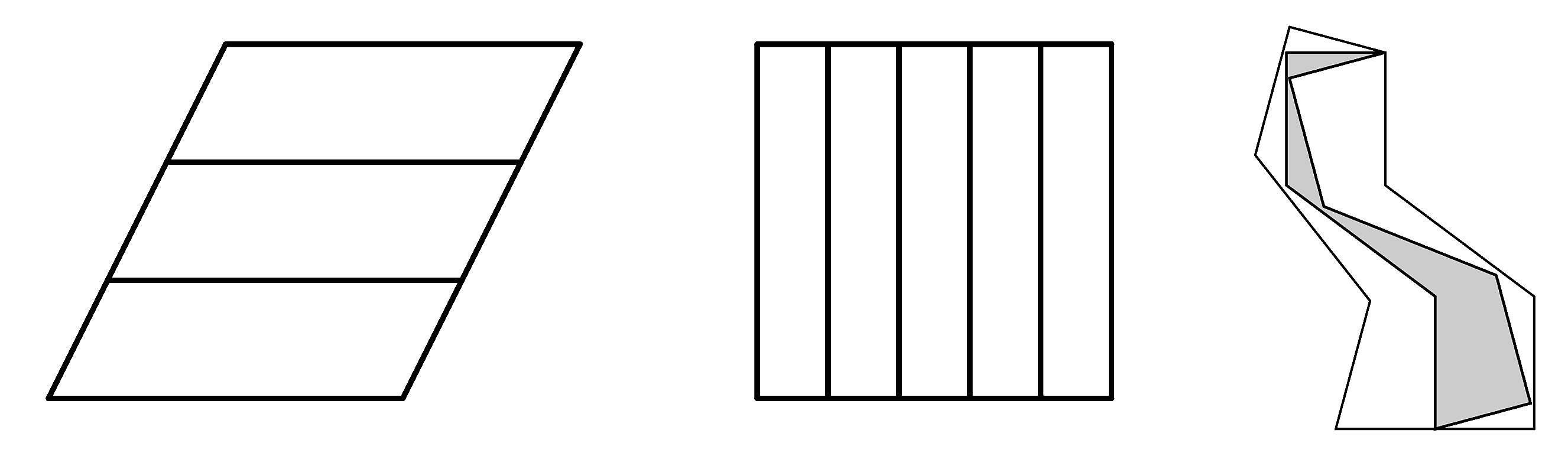}
\caption{\label{Nondiff} Some trivial examples of monohedral tilings of discs with non-differentiable boundaries, and a non-normal tiling (\cite{ALMT}).}
}
\end{figure}

In this paper, we deal with a generalisation of the results in \cite{KLV2020}. Of the distinctive properties of a circular disc, we preserve differentiability and strict convexity. 

In accordance with the definition of \cite[pp. 121]{GSTiling}, we say that a monohedral tiling with topological discs is {\it normal} if the intersection of any two tiles is either empty or connected.
A construction of a non-normal tiling, where the pieces are called {\it Voderberg-tiles}, can be found in \cite{ALMT}. Here, the intersection of the two outer tiles consists of two distinct points, see Figure~\ref{Nondiff}. In \cite{Mann2002}, Mann modifies this tiling in a way that the intersection consists of two disjoint non-degenerate curves, and the middle tile is contained in the interior of the tiled shape.


The main part of the paper pertains to normal tilings, while in Section~\ref{NonNormal}, we present some partial results in the case of non-normal tilings.

Throughout the proof, we denote the tiles of a tiling of $K$ by $\mathcal D_1, \ldots, \mathcal D_n$, the intersection $\mathcal D_i\cap \partial K$ by $\mathcal S_i$, and the isometry between $\mathcal D_i$ and $\mathcal D_j$ by $g_{ij}$,  i.e. $g_{ij}(\mathcal D_i)=\mathcal D_j$.

\begin{figure}[ht]
{\centering
\includegraphics[width=0.75\textwidth]{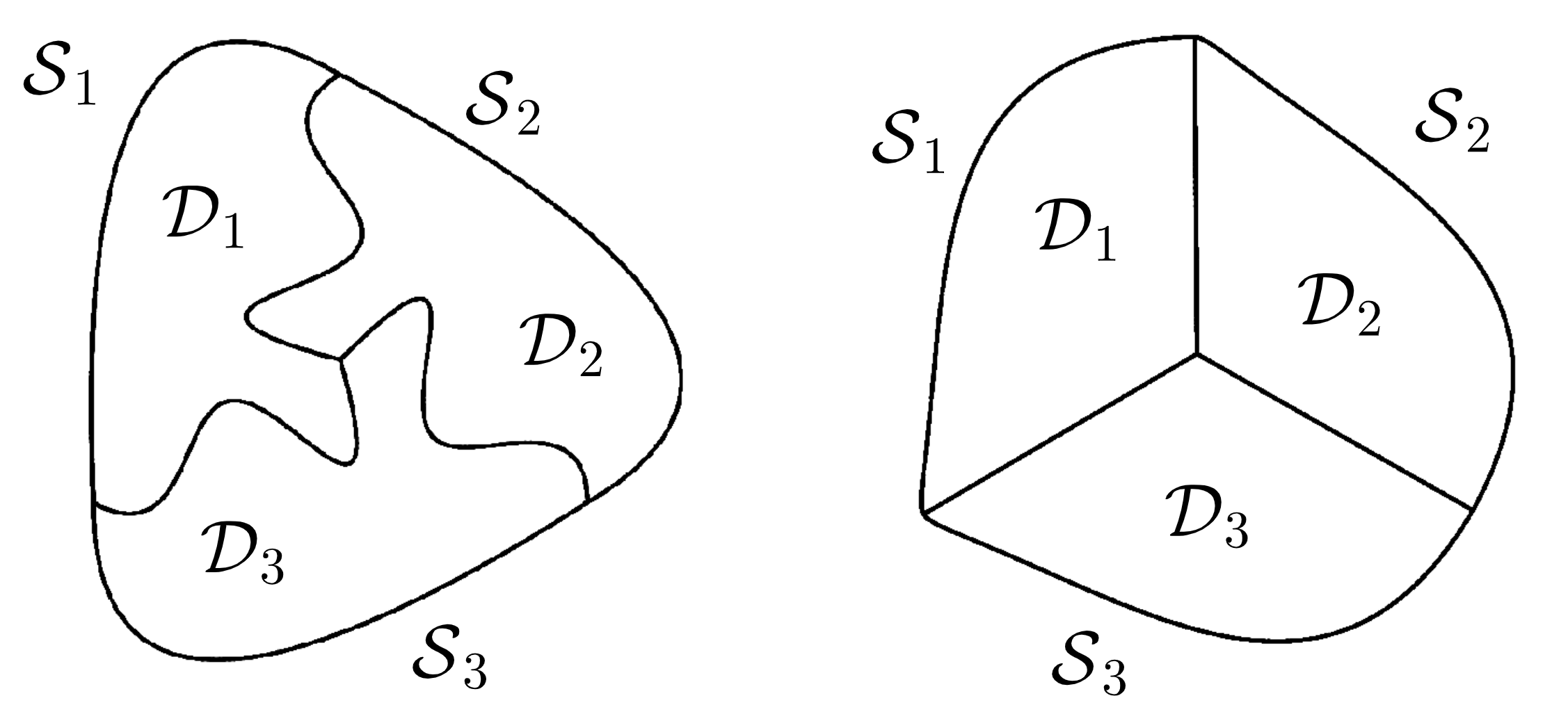}
\caption{\label{Constructions} Some normal monohedral tilings with three tiles.}
}
\end{figure}

Our main results are the following.

\begin{theorem}\label{two}
A smooth, strictly convex disc has a monohedral tiling with two tiles if and only if it has central or axial symmetry. The isometry between the tiles is the corresponding point or line reflection. 
\end{theorem}

\begin{theorem}\label{three}
Let $K$ be a strictly convex disc with a continuously differentiable boundary, and $\{ \mathcal D_1, \mathcal D_2, \mathcal D_3 \}$ be a monohedral, normal tiling of $K$. Assume that the boundary of the tile $\mathcal D_i$ is piecewise differentiable for $i=1,2,3$. 
Then the curves $\mathcal S_i$ $(i=1,2,3)$ are congruent. Moreover, the curves $\mathcal D_i\cap \mathcal D_j$ $(1\leq i <j \leq 3)$ are rotated copies of each other with an angle of $120^\circ$. 
\end{theorem}

\section{Notations and Tools}

In this paper, by a {\it simple curve} we mean a continuous map of an interval $[a, b]$ to the Euclidean plane. By the length of a differentiable curve with parametrisation $\mathbf r:[a,b]\to\mathbb{R}^2$, we mean the quantity $\int_a^b \vert \dot{\mathbf r}(t)\vert dt$, and denote it by $\length {\cdot}$. We say that a curve $\gamma$ is {\it non-degenerate} if $\length{\gamma}>0$.
We denote the boundary of a topological disc $K$ by $\partial K$, and its convex hull by $\conv K$. 
For a simple curve $\gamma$, we use the notation $\relint \gamma$ for its relative interior, that is, the set of points obtained from $\gamma$ by omitting its endpoints.

\begin{defin}
Let $K$ be a topological disc with a piecewise differentiable boundary, and $x\in \partial K$ a point on the boundary. 
Consider a parametrisation of $\partial K$ in a neighbourhood of $x$ such that $\mathbf r (0)=x$.
We say that $x$ is a {\it cusp} on the boundary of $K$ if the left and right tangent vectors $\dot{\mathbf r}_-(0)$ and $\dot{\mathbf r}_+(0)$ are parallel and have opposite directions. If the left tangent vector $\dot{\mathbf r}_-(0)$ points away from $K$, we say that $x$ is a {\it positive} cusp, otherwise we say it's a {\it negative} cusp.

We say that a non-differentiable point is {\it ordinary} if it is not a cusp. By the {\it angle} of an ordinary non-differentiable point, we mean the directed angle between the left and right tangents (in that order). An ordinary non-differentiable point is {\it positive} if its angle is positive in a parametrisation of the boundary, and {\it negative} otherwise.
\end{defin}

For further notation, we refer to \cite{DifferentialGeometry}.

\begin{defin}[\cite{KLV2020}]
A {\it multicurve} is a finite family of simple curves, called the {\it members of the multicurve}, which are parameterized on
non-degenerate closed finite intervals, and any point of the plane belongs to
at most one member, or it is the endpoint of exactly two members. If $\mathcal F$ and $\mathcal G$ are multicurves, $\bigcup \mathcal F=\bigcup \mathcal G$, and every member of $\mathcal F$ is the union of some members of $\mathcal G$, we say that $\mathcal G$ is a {\it partition} of $\mathcal F$. 
\end{defin}

\begin{defin}[\cite{KLV2020}]
Let $\mathcal F$ and $\mathcal G$ be multicurves. If there are partitions $\mathcal F'$ and
$\mathcal G'$ of $\mathcal F$ and $\mathcal G$, respectively, and a bijection $f: \mathcal F' \to \mathcal G'$, such that $f(\mathcal C)$ is
congruent to $\mathcal C$ for all $\mathcal C\in \mathcal F'$, we say that $\mathcal F$ and $\mathcal G$ are {\it equidecomposable.}
\end{defin}

\begin{lemma}[\cite{KLV2020}, p. 8, Corollary 2.14.]\label{equidecomp}
If $\mathcal{F}$ and $\mathcal{G}$ are equidecomposable,
and their subfamilies $\mathcal{F}' \subseteq \mathcal{F}$
and $\mathcal{G}' \subseteq \mathcal{G}$ are equidecomposable,
then $\mathcal{F} \setminus \mathcal{F}'$ and $\mathcal{G} \setminus \mathcal{G}'$
are equidecomposable.\end{lemma}

The following two technical lemmas are also from \cite{KLV2020}. 

\begin{lemma}[\cite{KLV2020}, p. 4, Lemma 2.1.]\label{finite}
Let $\Gamma$ be a Jordan curve and $\mathcal C$ be a simple curve.
Then $\Gamma$ contains only finitely many congruent copies of $\mathcal C$ that are mutually disjoint, apart from possibly their endpoints.
\end{lemma}

\begin{remark*}
Obviously, Lemma~\ref{finite} also holds if $\Gamma$ is a simple curve.
\end{remark*}

\begin{lemma}[\cite{KLV2020}, p. 4, Lemma 2.3.]\label{threebasis}
Let $\{ \mathcal D_1, \mathcal D_2, \mathcal D_3 \}$ be a tiling of the topological disc $\mathcal D$. Assume that $\mathcal S_i$ is a nondegenerate simple curve for every $i=1,2,3$.
Then $\mathcal D_1 \cap \mathcal D_2 \cap \mathcal D_3$ is a singleton $\{ M \}$,
and for any $i \neq j$, $\mathcal D_i \cap \mathcal D_j$ is
a simple curve connecting $M$ and a point in $\partial \mathcal D$.
\end{lemma}


\begin{lemma}\label{intersection}
If $n\geq 2$, there exist indices $i\neq j$ such that $\length{g_{ij}(\mathcal S_i)\cap\mathcal S_j}>0$.
\end{lemma}

\begin{remark*}
If we have $\length{ g_{ij}(\mathcal S_i)\cap\mathcal S_j}>0$, then taking the image of this intersection under $g_{jk}$ we also get $\length{g_{jk} g_{ij}(\mathcal S_i)\cap g_{jk} (\mathcal S_j)} =\length{g_{ik}(\mathcal S_i)\cap g_{jk} (\mathcal S_j)}>0$.
\end{remark*}

\begin{proof}
By the remark above, it is enough to consider the relationship between the curves $\mathcal S_1, g_{21}(\mathcal S_2),\ldots, g_{n1}(\mathcal S_n)$.
Suppose for contradiction that these curves are disjoint apart from possibly their endpoints. 
Since $\mathcal S_i\subseteq \partial\conv \mathcal D_i$, we have $g_{i1}(\mathcal S_i)\subseteq \partial\conv\mathcal D_1$. 
This yields that $\length{\mathcal S_1}+\ldots+\length{\mathcal S_n}=\length{\partial K} \leq \length{\partial\conv \mathcal D_1}$.
This is a contradiction, as $\conv \mathcal D_1\subsetneq K$, thus the perimeter of $K$ is strictly greater than the perimeter of $\conv \mathcal D_1$.
\end{proof}

\section {Proof of Theorem~\ref{two}.}


The {\it if} part of the statement is obvious. 
In the proof of the {\it only if} part, we denote the curve $\mathcal D_1 \cap \mathcal D_2$ by $\mathcal I$, and its two endpoints by $A$ and $B$, which are on $\partial K$.
Since $\mathcal D_1$ and $\mathcal D_2$ are congruent, their boundaries, $\mathcal S_1 \cup \mathcal I$ and $\mathcal S_2 \cup \mathcal I$ are equidecomposable, and hence by Lemma~\ref{equidecomp}, $\mathcal S_1$ and $\mathcal S_2$ are equidecomposable as well. As $\partial K$ is smooth, $\mathcal S_1$ and $\mathcal S_2$ are of equal length.
By Lemma~\ref{intersection}, the arc length of $\mathcal S_1\cap g_{21}(\mathcal S_2)$ is strictly positive. 
We consider two cases.

\noindent {\bf Case 1} $\; \mathcal S_1 \neq g_{21}(\mathcal S_2)$

Let's consider the position of the points $g_{21}(A)$ and $g_{21}(B)$ on $\partial\mathcal D_1$.

$g_{21}(A),\, g_{21}(B)\in \mathcal I$ yields $\mathcal S_1\subsetneq g_{21}(\mathcal S_2)$, which is a contradiction, since $\length{\mathcal S_1}=\length{\mathcal S_2}$. 

$g_{21}(A),\, g_{21}(B)\in \mathcal S_1$ either yields $\mathcal S_1\supsetneq g_{21}(\mathcal S_2)$, which is a contradiction similarly to the previous case, or $\mathcal I \subseteq g_{21}(\mathcal S_2)$. However, $\mathcal I \subseteq g_{21}(\mathcal S_2)$ means that $\mathcal D_1$ is stricly convex, while $\mathcal D_2$ is not, a contradiction.

The only remaining case is when one of the points $g_{21}(A)$ and $g_{21}(B)$ is an interior point of $\mathcal I$, and the other is an interior point of $\mathcal S_1$. 
Assume without loss of generality that $g_{21}(A)\in \relint \mathcal S_1$ and $g_{21}(B)\in \relint \mathcal I$. This means that $g_{21}(A)$ is differentiable on the boundary of $\mathcal D_1$, thus $A=g_{12}(g_{21}(A))$ is differentiable on $\partial \mathcal D_2$. 

If $g_{21}$ is orientation-preserving, then $A\in \relint g_{21}(\mathcal S_2)$. This yields that $A$ is differentiable on both $\partial \mathcal D_1$ and $\partial \mathcal D_2$, a contradiction.

If $g_{21}$ is orientation-reversing, then $B\in \relint (g_{21}(\mathcal S_2))$.
Let $\mathbf r:[0,1]\to \mathcal I$ be a continuous parametrisation of $\mathcal I$ such that $\mathbf r(0)=A$ and $\mathbf r(1)=B$. 

\begin{figure}[ht]
{\centering
\includegraphics[width=0.75\textwidth]{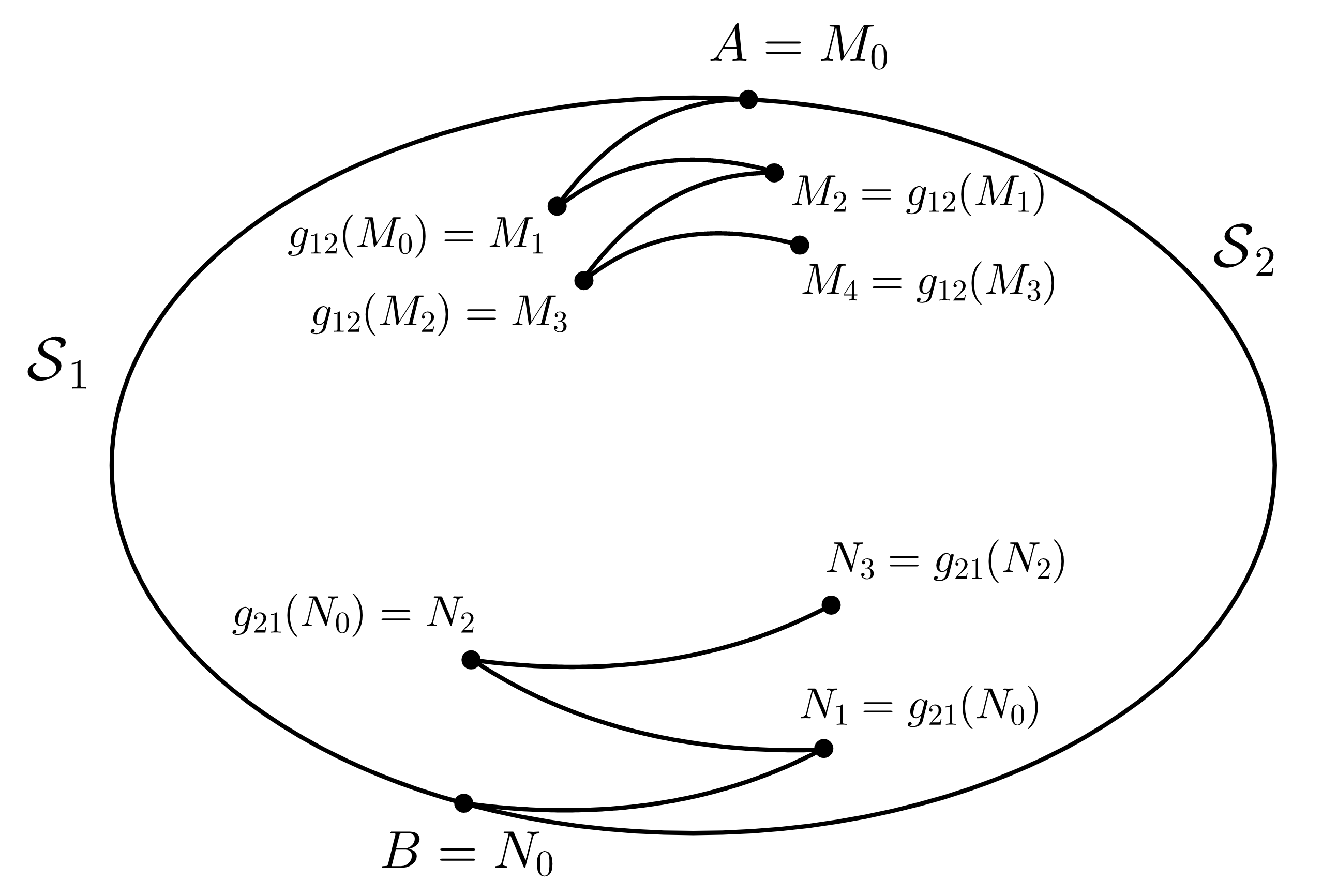}
\caption{\label{ZigzagTwoTiles} }
}
\end{figure}

Since $g_{21}(\mathcal S_2)$ is convex on $\partial \mathcal D_1$, and $g_{12}(\mathcal S_1)$ is convex on $\partial \mathcal D_2$, thus concave on $\partial \mathcal D_1$, their intersection is at most one point. Hence we have $\mathbf r^{-1}(g_{12}(A)) \leq \mathbf r^{-1} (g_{21}(B))$.
Since $M_1=g_{12}(A)\in \mathcal I\subset \partial\mathcal D_1$, we can consider its image under $g_{12}$. Because $M_1$ is non-differentiable, $M_2=g_{12}(M_1)=g_{12}^2(A)$ is non-differentiable, thus $M_2\in \mathcal I$, and we have $\mathbf r^{-1}(M_1)< \mathbf r^{-1}(M_2)$. Now, we can consider the image of $M_2$ under $g_{12}$, inducing a point $M_3=g_{12}(M_2)=g_{12}^3(A)$ that is also on $\mathcal I$. Continuing this process, let $M_i=g_{12}(M_{i-1})$, where for every index $i$ we have $M_i\in\mathcal I$ and $\mathbf r^{-1}(M_{i-1})< \mathbf r^{-1}(M_i)$
Similarly, we are able to obtain points $N_1=g_{21}(B)$ and $N_j=g_{21}(N_{j-1})$, see Figure~\ref{ZigzagTwoTiles}.

In the parametrisation $\mathbf r$ of $\mathcal I$, the end point of an arbitrary arc $\widehat{M_iM_{i+1}}$ is a positive cusp on $\partial\mathcal D_1$ if and only if $\widehat{M_iM_{i+1}}$ is convex on $\partial\mathcal D_1$.
Conversely, the end point of an arbitrary arc $\widehat{N_jN_{j-1}}$ is a positive cusp on $\partial\mathcal D_1$ if and only if $\widehat{N_jN_{j-1}}$ is concave on $\partial\mathcal D_1$. 
Thus there are no $i,j$ indices such that $M_i=N_j$. This yields that there are infinitely many $M_i$ and $N_j$ points, and hence infinitely many arcs on $\mathcal I$ congruent to $\widehat{AM_1}$ and $\widehat{BN_1}$, which is a contradiction by Lemma~\ref{finite}.

\noindent {\bf Case 2} $\;\mathcal S_1=g_{21}(\mathcal S_2)$

If $g_{21}(A)=A$, then $g_{21}(B)=B$, and $g_{21}$ is the reflection about the line $AB$. This yields that $K$ is axially symmetric.

If $g_{21}(A)=B$, then $g_{21}(B)=A$, and $g_{21}$ is the reflection about the midpoint $O$ of the segment $\overline{AB}$. This yields that $K$ is centrally symmetric about $O$.

\section {Proof of Theorem~\ref{three}.}

\begin{lemma}
Every non-empty $\mathcal S_i$ $(i=1,2,3)$ is either the union of at most two connected simple curves, or empty.
\end{lemma}

\begin{figure}[h]
{\centering
\includegraphics[width=0.6\textwidth]{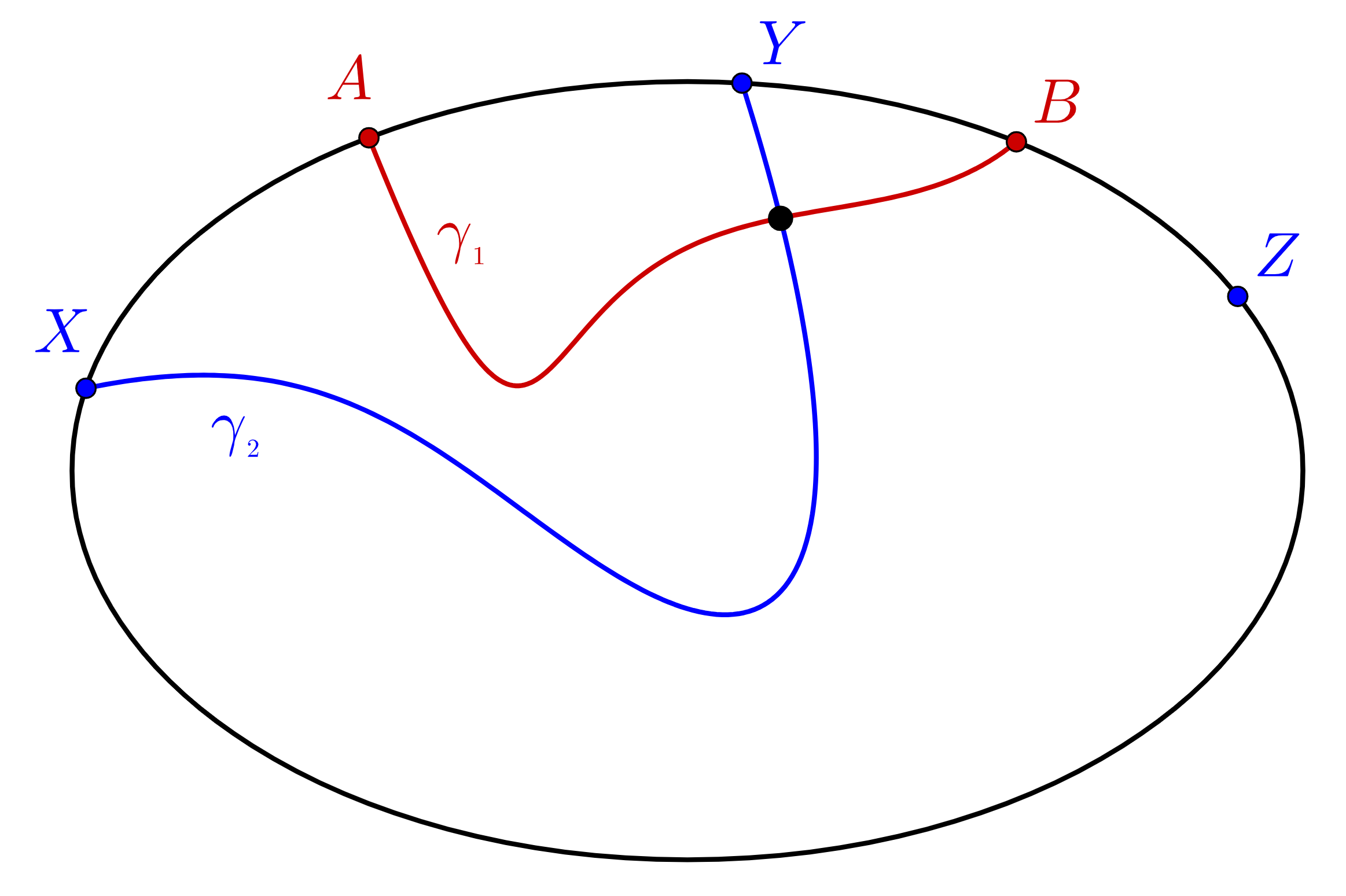}
\caption{\label{Multipiece} $\mathcal{S}_i$ consists of at most two components.}
}
\end{figure}

\begin{proof}
Suppose that an $\mathcal S_i$, say $\mathcal S_2$, consists of at least three components, and choose points $X,\, Y,\, Z\in\mathcal S_2$ from three different components.
Suppose $\mathcal S_1$ contains points $A, B$ such that $A$ is on the curve $\widehat{XY}\subseteq \partial K$ not containining $Z$, and $B$ is on $\widehat{YZ}\subseteq \partial K$ not containing $X$.
Let $\gamma_1\subset\inter \mathcal D_1$ be a curve joining $A$ and $B$, and $\gamma_2\subset \inter \mathcal D_2$ a curve joining $X$ and $Y$, see Figure~\ref{Multipiece}.
Given the continuity of the curves, $\gamma_1\cap\gamma_2\neq\emptyset$, hence $\inter \mathcal D_1 \cap \inter \mathcal D_2\neq \emptyset$, a contradiction.
Thus $\mathcal D_i$ $(i\neq 2)$ only contains points from one of the curves $\widehat{XY},\, \widehat{YZ},\, \widehat{ZX}$ hence we need at least three tiles apart from $\mathcal D_2$ to cover the boundary of $K$.
Since we have three tiles altogether, this yields a contradiction, and $\mathcal S_2$ consists of at most two connected curves.
\end{proof}

As a result of the lemma, we distinguish three topological cases, see Figure~\ref{TopCases}.
If an $\mathcal S_i$ consists of two components, the other two are necessarily non-empty and connected. We label this case I. 
If every $\mathcal S_i$ is connected, at most one of them is empty by Lemma~\ref{intersection}: in case II, neither is empty, and in case III, exactly one is empty. 
We note that the tilings of case III are not normal tilings, along with some extreme subcases of cases I and II. (We also note that technically there should be a fourth case when the complete boundary of $K$ is covered by a single tile, but it trivially cannot occur.)

\begin{figure}[ht]
\includegraphics[width=\textwidth]{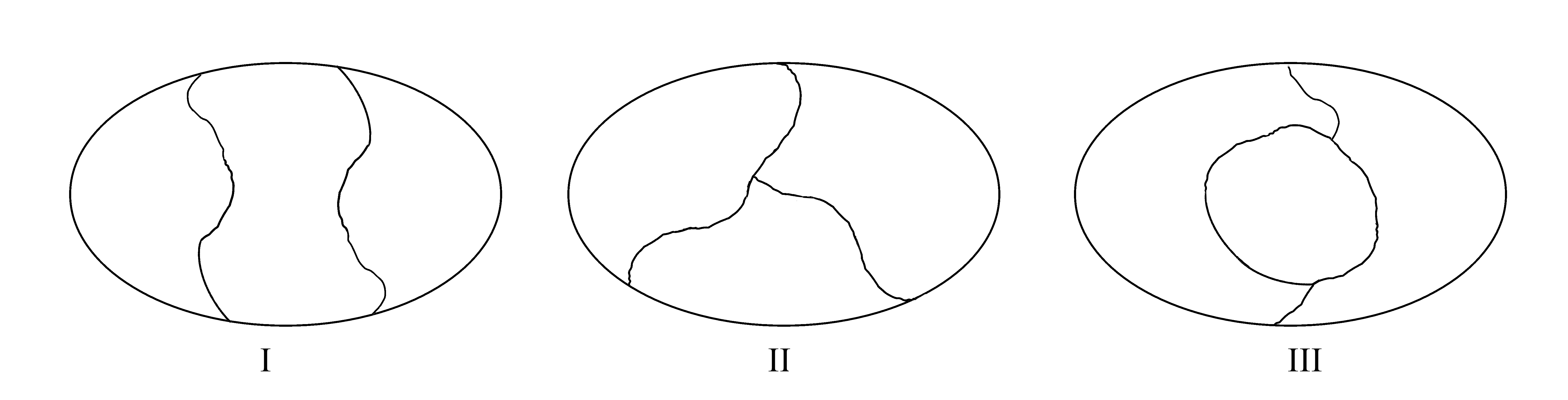}
\caption{\label{TopCases} The topological cases of monohedral tilings with three tiles.}
\end{figure}


\begin{lemma}\label{connected}
Every $\mathcal S_i$ $(i=1,2,3)$ is connected in every monohedral tiling with three tiles.
\end{lemma}

\begin{proof}
Suppose for contradiction that an $\mathcal S_i$, say $\mathcal S_2$, is not connected. Hence by Lemma~\ref{Multipiece}, $\mathcal S_2$ consists of exactly two components, denoted by $\mathcal S_2'$ and $\mathcal S_2''$. We use the notation $\mathcal I_1=\mathcal D_1\cap \mathcal D_2$ and $\mathcal I_2=\mathcal D_2\cap \mathcal D_3$, see Figure~\ref{NotConnected}.
As $\partial\mathcal D_1=\mathcal S_1\cup\mathcal I_1$ and $\partial \mathcal D_2=\mathcal I_1\cup\mathcal S_2\cup\mathcal I_2$ are congruent, we obtain by Lemma~\ref{equidecomp} that $\mathcal S_1$ and $\mathcal S_2\cup\mathcal I_2$ are equidecomposable. This yields that $\mathcal I_2$ is differentiable except for a finite number of connecting points, which is similarly true for $\mathcal I_1$.

\begin{figure}[ht]
{\centering
\includegraphics[width=0.55\textwidth]{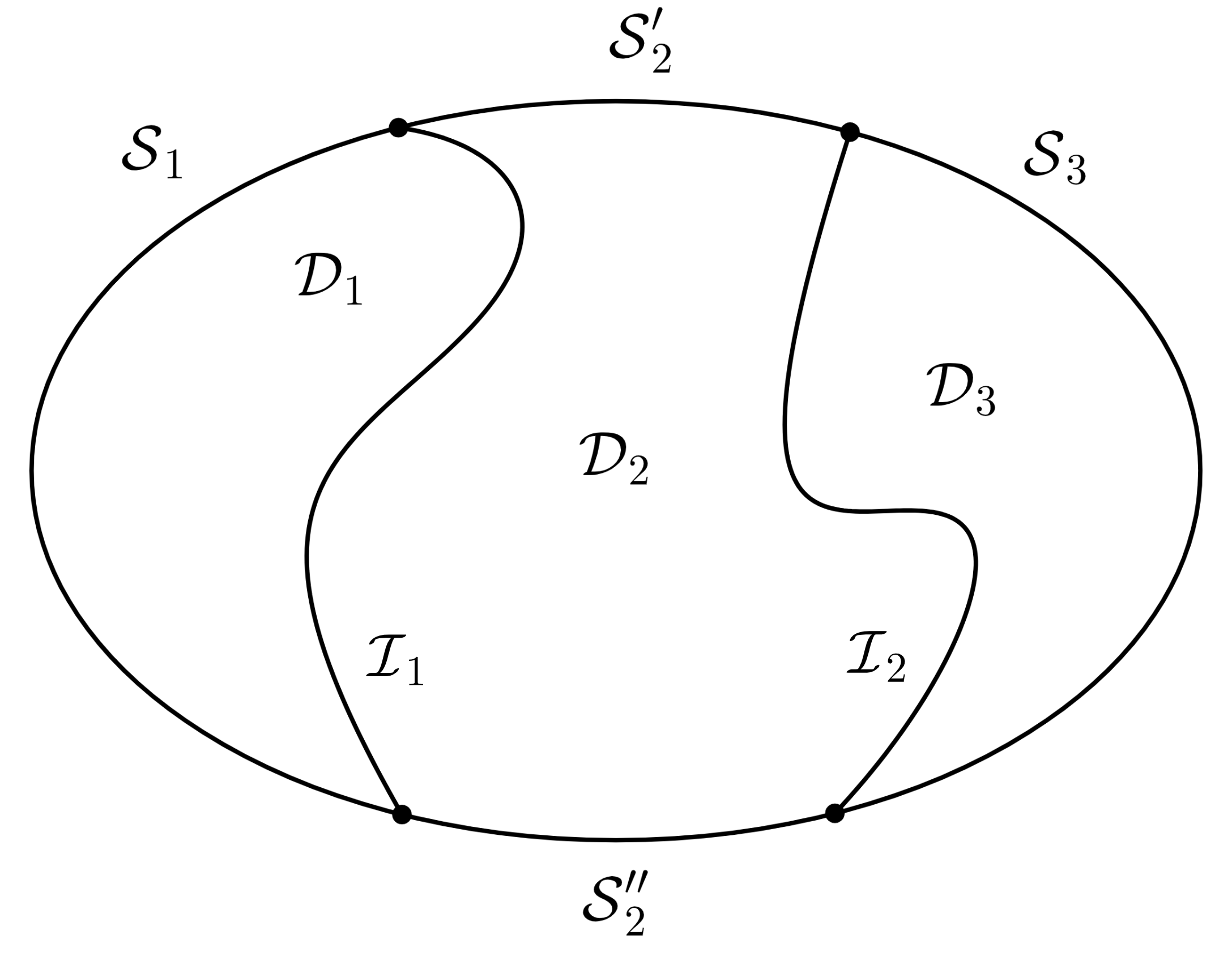}
\caption{\label{NotConnected} The case where an $\mathcal S_i$ is not connected does not occur.}
}
\end{figure}

As $\mathcal S_1$ is differentiable apart from possibly its endpoints, we can say that every non-differentiable boundary point of $\mathcal D_1$ is on $\mathcal I_1$, say their number is $k$. Out of these $k$ points, at least $k-2$ are on $\relint \mathcal I_1$. Similarly, $\mathcal I_2$ also contains $k$ non-differentiable points on $\partial\mathcal D_3$, of which at least $k-2$ are on $\relint \mathcal I_2$. 
The non-differentiable points on $\relint \mathcal I_1 \subseteq\partial\mathcal D_1$ and $\relint \mathcal I_2\subseteq \partial\mathcal D_2$ are also non-differentiable on the boundary of $\mathcal D_3$. This yields that $2k-4\leq k \leq 2k$, and we have $0\leq k \leq 4$. As each endpoint of $\mathcal I_1$ and $\mathcal I_2$ is non-differentiable on the boundary of at least one tile it's contained in, we have $k\geq 2$.

In the case where $k=2$, the calculation yields that both non-differentiable points of $\partial\mathcal D_1$ and $\partial\mathcal D_3$ are the endpoints of $\mathcal I_1$ and $\mathcal I_2$, respectively. Hence, the non-differentiable points of $\partial\mathcal D_2$ are also on $\partial K$.
As these points are also non-differentiable on the boundaries of the other tiles they're contained in, they are necessarily ordinary non-differentiable points on the boundary of $\mathcal D_2$. 
If at least one of the curves $\mathcal S_2'$ and $\mathcal S_2''$ is non-degenerate, the endpoints of $\mathcal I_1$ and $\mathcal I_2$ determine three or four points, at least one of which is differentiable on $\partial\mathcal D_2$. This induces a cusp on the boundary of either $\mathcal D_1$ or $\mathcal D_3$, which is a contradiction, as $\mathcal D_2$ only contains ordinary non-differentiable boundary points.

If $k=3$, $\relint \mathcal I_1$ and $\relint \mathcal I_2$ each contain exactly one non-differentiable point. The endpoints of $\mathcal I_1$ and $\mathcal I_2$ are non-differentiable on $\partial\mathcal D_1$ and $\partial \mathcal D_3$, and exactly one of them is non-differentiable on $\partial \mathcal D_2$. Assume this is an endpoint of $\mathcal I_2$.
Then this point is a ordinary non-differential point on both $\partial \mathcal D_2$ and $\partial \mathcal D_3$, so $\partial\mathcal D_1$ must contain one as well, which must be the one on $\relint \mathcal I_1$. This yields that $\partial\mathcal D_2$ only contains at most one cusp, a contradiction, since $\partial \mathcal D_1$ contains two.

If $k=4$, all endpoints of $\mathcal I_1$ and $\mathcal I_2$ are positive cusps on the boundary of $\mathcal D_1$ and $\mathcal D_3$, and differentiable on $\partial \mathcal D_2$. Apart from their endpoints, they both contain two non-differentiable points. 
If $\partial\mathcal D_1$ and $\partial\mathcal D_3$ each contain one or two ordinary non-differentiable points, $\partial \mathcal D_2$ contains, respectively, two or four, a contradiction. Hence every non-differentiable point is a cusp. 
It is easy to count that the number of positive and negative cusps can't be equal on the boundary of every tile. 
\end{proof}

\begin{remark}\label{NonNormalTopI}
The case where both $\mathcal S_2'$ and $\mathcal S_2''$ consist of a single point is a non-normal tiling. It is easy to see that the constraints for $k$ are the same, and $k=3$ and $k=4$ yield a contradiction in the same way as in the main part of the proof. In the case $k=2$, the endpoints of $\mathcal I_1$ coincide with the endpoints of $\mathcal I_2$, say their endpoints are $A$ and $B$, and they are both ordinary non-differentiable points on the boundary of all three tiles. 
As $A$ and $B$ are the only two non-differentiable points on the boundary of every tile, we have for every isometry $g_{ij}$ that either both $A$ and $B$ are fixed points, or $g_{ij}(A)=B$ and $g_{ij}(B)=A$. There are four isometries satisfying one of these requirements: the identity, the reflection about $AB$, the reflection about the perpendicular bisector of the segment $\overline{AB}$, and the point reflection about the midpoint of $\overline{AB}$. Every option leads to a contradiction. We leave the details to the reader.
\end{remark}

\begin{remark*}
The differentiability of the boundary is an important condition. It is easy to find a counterexample in the case of a non-differentiable boundary, see the parallelogram in Figure~\ref{Nondiff}. We also emphasise that due to the differentiability of $\partial K$, the boundaries of the tiles are automatically piecewise differentiable in topological case I, so it is not necessary to make that assumption in Lemma~\ref{connected}.
\end{remark*}


Now, we begin the proof of Theorem~\ref{three}. By Lemma~\ref{connected}, topological case I does not occur, and as the tilings of case III are not normal, we only have to consider case II. Hence we have that every $\mathcal S_i$ is a non-empty connected curve.

By Lemma~\ref{intersection}, there exist indices $i\neq j$ such that the intersection of $\mathcal S_j$ and $g_{ij}(\mathcal S_i)$ is a non-degenerate curve.
Assume without loss of generality that $\mathcal S_1\cap g_{21}(\mathcal S_2)$ is a non-degenerate curve, which also yields that $\length{\mathcal S_1},\length{\mathcal S_2}>0$.

In the case of a normal tiling, $\length{\mathcal S_3}>0$. 
Then by Lemma~\ref{threebasis}, the intersection of the three tiles is a single point, say $M=\mathcal D_1\cap\mathcal D_2\cap\mathcal D_3$, in $\inter K$, and the intersection of any two tiles is a simple surve connecting $M$ and a point on $\partial K$. We denote the non-degenerate curve $\mathcal D_i\cap\mathcal D_j$ by $\mathcal I_{ij}$, and its endpoint on $\partial K$ by $A_{ij}$ $(1\leq i<j\leq 3)$.

\noindent {\bf Case 1} $g_{21}$ is orientation-preserving.

We present the proof in the case $M\notin g_{21}(\mathcal S_2)$. The case $M\in g_{21}(\mathcal S_2)$ can be proved similarly. The details are left to the reader.

Let us consider the convex differentiable curve $\mathcal S_1\cup g_{21}(\mathcal S_2)\subseteq \partial\mathcal D_1$. As $M$ is not on the curve, its two endpoints are on $\mathcal I_{13}$ and $\mathcal I_{12}$. We denote its endpoint on $\mathcal I_{13}$ by $X_1$. Let $X_2=g_{12}(X_1)$ and $X_3=g_{12}(X_2)$, where we have $X_2\in\mathcal I_{12}$ and $X_3\in\mathcal I_{23}$. 

\begin{figure}[ht]
\begin{center}
\includegraphics[width=0.65\textwidth]{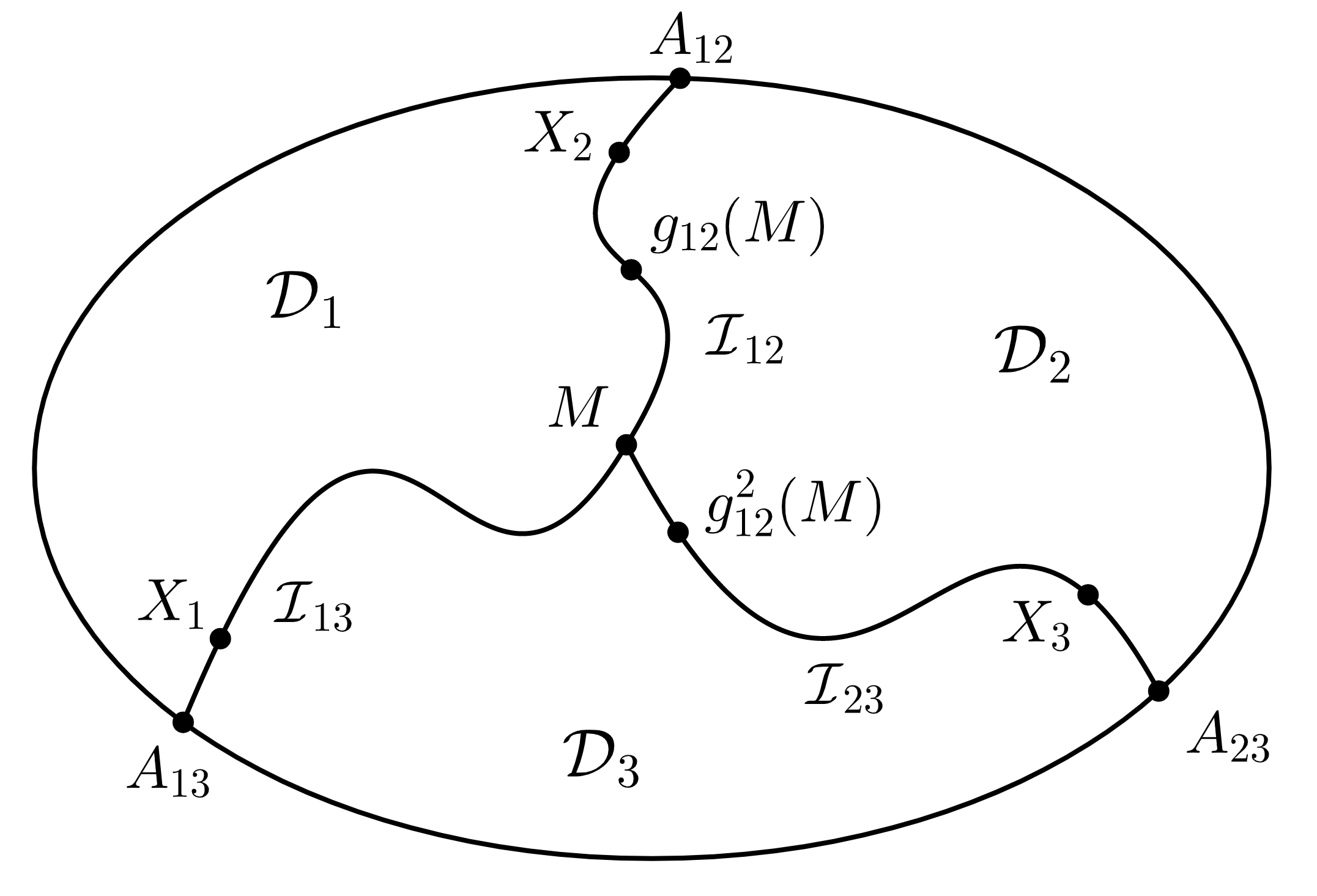}
\caption{\label{preservingMwrong}}
\end{center}
\end{figure}

First consider the case $g_{12}(M)\in\mathcal I_{12}\setminus\{M\}$. $g_{12}(M)$ is the unique point on $\widehat{MX_2}$ which is of distance $\length{\widehat{Mg_{12}(M)}}$ from $M$ on $\partial \mathcal D_1$, hence $g_{12}^2(M)$ is the unique point on $\widehat{g_{12}(M)X_3}$ which is of distance $\length{\widehat{Mg_{12}(M)}}$ from $g_{12}(M)$ on $\partial \mathcal D_2$. This point is exactly $M$, so $g_{12}^2(M)=M$.  
This yields that $g_{12}$ is either a rotation around $M$, or a reflection about the midpoint of $\overline{Mg_{12}(M)}$. The former doesn't fit with our assumption that $g_{12}(M)\neq M$, and the latter yields $X_3=g_{12}^2(X_1)=X_1$, a contradiction.

From now on, we consider the case $g_{12}(M)\in \mathcal I_{23}$. 
Let $M_1=g_{21}(M)$ and $M_2=g_{12}(M)$, where we have $M_1\in\mathcal I_{13}$ and $M_2\in\mathcal I_{23}$.
Let $\mathcal C_i=\widehat{MM_i}\subseteq \mathcal I_{i3}$, and note that $\mathcal C_2=g_{12}(\mathcal C_1)$.
We define the curves $\mathcal J_1=\widehat{X_1M_1}\subseteq\mathcal I_{13}$, $\mathcal J_2=\widehat{X_2M}\subseteq\mathcal I_{12}$, $\mathcal I_3=\widehat{X_3M_2}\subseteq\mathcal I_{23}$, and note that $\mathcal J_2=g_{12}(\mathcal J_1)$ and $\mathcal J_3=g_{12}(\mathcal J_2)$.

For the case $g_{12}(M)\in\relint \mathcal I_{23}$, we consider five subcases, based on the position of $g_{12}(\mathcal S_1)$ on $\partial\mathcal D_2$. 

\begin{figure}[ht]
\includegraphics[width=\textwidth]{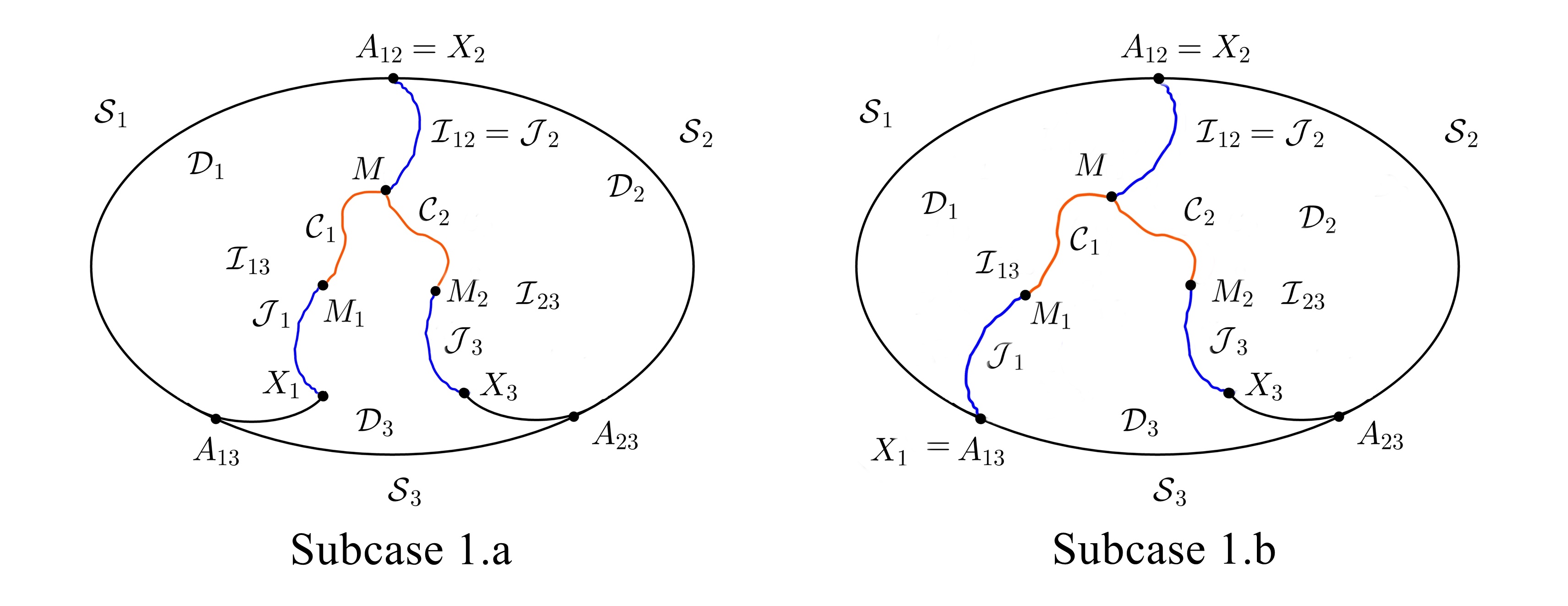}
\caption{\label{preservingCasesab}}
\end{figure}

\noindent{\bf Subcase 1.a} $\; g_{12}(A_{12})\in\relint (\mathcal I_{23})$ and $g_{12}(A_{13})\in\relint(\mathcal S_2)$. (See Figure~\ref{preservingCasesab}.)

We consider the non-differentiable points of the boundaries of the tiles.
$g_{12}(A_{13})\in\relint \mathcal S_2$ yields $g_{21}(A_{12})\in\relint \mathcal I_{13}$, thus $A_{13}$ and $A_{23}$ are differentiable on $\partial\mathcal D_1$ and $\partial\mathcal D_2$, respectively, hence they are cusps on $\partial\mathcal D_3$.

If $\relint \mathcal J_1$ contains, say, $m$ non-differentiable points on the boundary of $\mathcal D_3$, it contains exactly the same $m$ on $\mathcal D_1$. Moreover, as $\mathcal J_2$ and $\mathcal J_3$ are congruent to $\mathcal J_1$, the number of non-differentiable points on $\relint \mathcal J_2$ and $\relint \mathcal J_3$ is $m$ as well. This yields that every tile obtains exactly $2m$ non-differentiable points from the interior points of $\mathcal J_i$. 
Additionally, if an arbitrary point $A\in\relint \mathcal J_1$ is a positive non-differentiable point on $\partial\mathcal D_3$, it's negative on $\partial \mathcal D_1$, $g_{12}(A)$ is positive on $\partial\mathcal D_1$, and $g_{12}^2(A)$ is negative on $\partial\mathcal D_3$. Hence the number of positive and negative non-differentiable points on $\relint \mathcal J_i$ coincide on the boundary of every tile. 
Based on these assertions, we may disregard the non-differentiable points of $\relint \mathcal J_i$ from now on.

Let $\mathcal C_1$ and $\mathcal C_2$ contain $k$ non-differentiable points apart from their endpoints. Then $\partial \mathcal D_1$ and $\partial \mathcal D_3$ contain $k$ and $2k+2$ non-differentiable points, respectively, apart from the points $M,\, M_i, \, X_i$.
For $k$ to have a non-negative value, the properties of thed $M_i$-s and $X_i$-s must be so that the number of non-differentiable points on $\partial\mathcal D_1$ increases by at least two compared to their number on $\partial\mathcal D_3$.

As the image of $X_2=A_{12}$ under $g_{12}$ and $g_{21}$ is on $\mathcal I_{23}$ and $\mathcal I_{13}$, if $X_2$ is not differentiable on the boundary of one or two tiles it's contained in, every tile obtains exactly one or two non-differentiable points from the points $X_i$.
If $M$ is non-differentiable on $\partial\mathcal D_1$, then $M_2$ is non-differentiable on $\partial\mathcal D_3$.

\noindent {\bf Subcase 1.b} $\; g_{12}(A_{12})\in\relint\mathcal I_{23}$ and $g_{12}(A_{13})=A_{12}$.

Using the notation in Subcase 1.a, $\partial\mathcal D_1$ and $\partial \mathcal D_3$ contain, respectively, $k$ and $2k+1$ non-differentiable points apart from $M,\,M_i,\, X_i$, thus we have to increase the number of non-differentiable points on the boundary of $\mathcal D_1$ compared to $\mathcal D_3$.

Similarly to the previous case, $\partial\mathcal D_1$ can't obtain more non-differentiable points from $M$ and its images than $\partial\mathcal D_3$ does, and the same is true if $A_{13}$ is non-differentiable on $\partial\mathcal D_3$.
Thus $A_{13}$ must be differentiable on $\partial \mathcal D_3$, which induces a positive cusp on $\partial\mathcal D_1$. 
This yields $k=0$, and as we can't increase the number of non-differentiable points of $\partial\mathcal D_1$ further, $M$ must be differentiable on the boundary of $\mathcal D_3$.
Now, we have one positive cusp on the boundary of each tile, thus $M$ must have properties such that every tile obtains the same number of positive, and the same number of negative non-differentiable points from $M,\, M_1, \, M_2$. 
As $M$ is differentiable on $\partial \mathcal D_3$, if it's not differentiable on the boundary of $\mathcal D_1$ $(\mathcal D_2)$, then it's a positive non-differentiable point. Then $M_2$ $(M_1)$ is positive on $\partial\mathcal D_2$ $(\partial\mathcal D_1)$, and negative on $\partial\mathcal D_3$. This is a contradiction, because neither $M$, nor $M_1$ can be negative on $\partial\mathcal D_1$. 

\noindent {\bf Subcase 1.c} $\; g_{12}(A_{12})=A_{23}$ and $g_{12}(A_{13})\in\relint(\mathcal I_{12})$. (See Figure~\ref{preservingCasescde}.)

In this case, we consider not only the number of non-differentiable points of the boundaries, but also whether they are positive or negative. Say $\relint(\mathcal C_1)$ contains $m$ positive and $l$ negative cusps on $\partial\mathcal D_1$. 
As the differentiability of the points $X_i$ is independent from the differentiability of $M_i$-s, we consider eight different cases. The number of positive/negative cusps on $\partial \mathcal D_1$ is equal to the number of positive/negative cusps on $\partial\mathcal D_3$, this yields two equations in $m$ and $l$.
The two acquired equations are independent in all cases, but the values of $m$ and $l$ are either negative or not integers, a contradiction. The detailed computations can be found in Table~\ref{1cdetails}.

\begin{table}[h]
{\centering
\begin{tabular}{||c|cc|cc||}
\hline\hline
&$\partial\mathcal D_1$& &$\partial\mathcal D_3$&\\
& positive & negative & positive & negative\\
\hline
Aside from points $X_i,M_i$& $m+1$ & $l$ & $2l$ & $2m$\\
\hline
$X_1$ cusp on $\partial\mathcal D_1$ & +1 & +1 & +1 & 0\\
$X_1$ cusp on $\partial\mathcal D_3$ & 0 & 0 & +1 & 0\\
\hline
$M$ is differentiable on $\partial\mathcal D_1,\partial\mathcal D_2$ & 0 & 0 & +1 & 0\\
$M$ is differentiable on $\partial\mathcal D_1$, cusp on $\partial\mathcal D_2$ & +1 & 0 & 0 & +1 \\
$M$ positive cusp on $\partial\mathcal D_1,\partial\mathcal D_2$ & +2 & 0 & 0 & +3 \\
$M$ positive cusp on $\partial\mathcal D_1$, negative cusp on $\partial\mathcal D_2$ & +1 & +1 & +2 & +1\\
\hline\hline
\end{tabular}
\caption{\label{1cdetails} Further subcases of Subcase 1.c.}
}
\end{table}

\begin{figure}[ht]
\includegraphics[width=\textwidth]{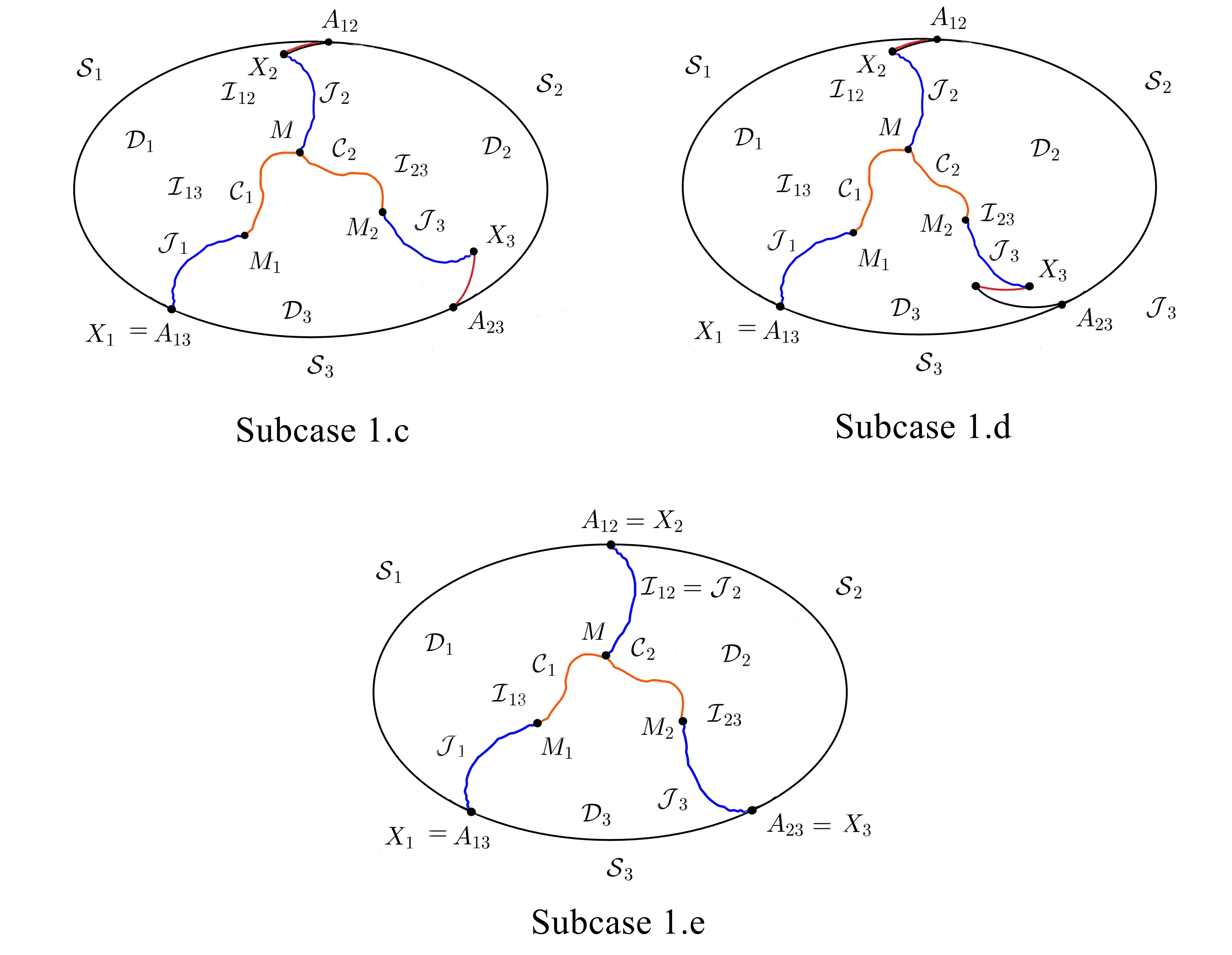}
\caption{\label{preservingCasescde}}
\end{figure}

\noindent {\bf Subcase 1.d} $\; g_{12}(A_{12})\in\mathcal I_{23}$ and $g_{12}(A_{13})\in\mathcal I_{12}$

In this case, we have that $A_{12}$ is a positive cusp on $\partial\mathcal D_1$ and $A_{23}$ is a positive cusp on $\partial\mathcal D_3$. Furthermore, $g_{12}(A_{12})\in\mathcal I_{23}$ is a positive cusp on $\partial\mathcal D_2$, thus a negative cusp on $\partial\mathcal D_3$. This yields that $\partial\mathcal D_1$ and $\partial\mathcal D_3$ contain, respectively, $k+1$ and $2k+2$ non-differentiable points apart from the points $X_i$ and $M_i$. Once again, we must choose the points $X_i$ and $M_i$ in a way that increases the number of non-differentiable points on $\partial \mathcal D_1$ compared to $\partial\mathcal D_3$.

If $X_1$ is a cusp on $\partial\mathcal D_3$, $X_1$ and $X_2$ are differentiable on $\partial\mathcal D_1$, so it does not obtain non-differentiable points from $X_i$-s.
If $X_1$ is an ordinary non-differentiable point on the boundary of both tiles, every tile gains exactly two non-differentiable points, hence their number doesn't change relative to each other. 
The only remaining case is when $X_1$ is a cusp on $\partial\mathcal D_1$, hence its number of non-differentiable points increases by exactly one more than on $\partial\mathcal D_3$. Thus $k=0$, $M$ is differentiable on $\partial\mathcal D_3$, and all three tiles contain two positive and one negative cusp aside from $M$ and $M_i$. Hence we must choose the differentiability of $M$ in a way that every tile gets the same number of positive, and the same number of negative non-differentiable points. 
As we've seen in the previous cases, this leads to a contradiction. 

\noindent {\bf Subcase 1.e} $\; g_{12}(A_{12})=A_{23}$ and $g_{12}(A_{13})=A_{12}$

$\partial \mathcal D_1$ and $\partial\mathcal D_3$ contain $k$ and $2k$ non-differentiable points apart from $M,\, M_i, \, X_i$.
The number and type of non-differential points obtained from points $X_i$ is the same on every tile, which yields that it has to be the same on $M,\, M_i$ as well. 
This leads to a contradiction.

This concludes the proof in the case $M_2 \in\relint \mathcal I_{23}$. 
Now, we consider the case $M=M_1=M_2$. We clearly have that $g_{12}$ is a rotation around $M$.

The position of $g_{12}(\mathcal S_1)$ on $\partial\mathcal D_2$ can be the same five types as previously. Using a similar argument as before, it is easy to see that $M$ must be an ordinary non-differential point on the boundary of each tile. More specifically, the angle of $M$ is the same on each tile, thus $120^\circ$, which yields that the curves $\mathcal J_{ij}$ $(1\leq i<j \leq 3)$ can be obtained from each other through a $120^\circ$ rotation around $M$.  
Now, consider every ordinary non-differentiable point of the boundaries. Assume $\mathcal I_{13}$ contains $n$ such points, where $n\geq 1$ due to the non-differentiability of $M$. Label these points $P_1, \ldots, P_n$, where $P_n=M$. $\mathcal I_{12}$ and $\mathcal I_{23}$ contain exactly $n$ ordinary non-differentiable points as well: points $g_{12}(P_i)$ and $g_{12}^2(P_i)$. 
Hence the boundary of every tile is the union of $2n-1$ curves that are pairwise disjoint apart from their endpoints, and where every curve contains exactly two ordinary non-differentiable points, their endpoints.

As $\widehat{P_iP_{i+1}}\cong\widehat{g_{12}(P_i)g_{12}(P_{i+1})}\cong \widehat {g_{12}^2(P_i)g_{12}^2(P_{i+1})}$, every tile contains two curves congruent to $\widehat{P_iP_{i+1}}$. This yields that the components on $\partial\mathcal D_1$ and $\partial\mathcal D_3$ containing $\mathcal S_1$ and $\mathcal S_3$, are congruent as well. 
As $\widehat{X_1P_1}\cong \widehat{X_2g_{12}(P_1)} \cong \widehat{X_3g_{12}^2(P_1)}$, we have that $\widehat{X_1X_2}\subseteq \partial \mathcal D_1$ is congruent to the curve $\widehat{X_1X_3}\subseteq \partial \mathcal D_3$. It is easy to see that this yields a contradiction if $g_{12}(\mathcal S_1)\neq \mathcal S_2$ (previously Subcases a-d), as the two curves don't contain the same number of non-differentiable points.
In the case where $g_{12}(\mathcal S_1)=\mathcal S_2$ (Subcase e), we obtain that $\mathcal S_1 \cong \mathcal S_3$. We also have that $\mathcal I_{ij}=\mathcal J_{ij}$, which yields by our previous remarks that the curves $\mathcal I_{ij}$ $(1\leq i<j \leq 3)$ can be obtained from each other through a $120^\circ$ rotation around $M$.

\noindent {\bf Case 2} $\; g_{12}$ is orientation-reversing.

\noindent {\bf Subcase 2.a} $\; g_{12}$ is a glide reflection.

As $g_{12}$ is a glide reflection, it has no fixed points. Thus we either have $g_{12}(A_{12})\in\relint{\mathcal S_2}$ or $g_{12}(A_{12})\in \partial\mathcal D_2\setminus \mathcal S_2$. The latter yields $g_{21}(A_{12})\in\relint{\mathcal S_1}$, which is a case fundamentally similar to the first one. Thus we may assume without loss of generality that $g_{12}(A_{12})\notin\mathcal S_2$.

Let $M_0=A_{12}$, $M_1=g_{12}(M_0)$ and $M_{i+1}=g_{12}(M_i)$ if $M_i\in\mathcal I_{12}$. By Lemma~\ref{finite}, $\mathcal I_{12}$ contains a finite number of arcs congruent to $\widehat{M_0M_1}$, thus there exists an index $k$ for which $M_k\in\mathcal I_{12}$ and $M_{k+1} \notin \mathcal I_{12}$, hence $M$ is between $M_k$ and $M_{k+1}$ on $\partial\mathcal D_2$. It is possible that $M_k=M$.
This yields that $\mathcal I_{12}$ consists of $k$ differentiable arcs congruent to $\widehat{M_0M_1}$ joined by cusps, and the differentiable arc between $M_k$ and $M$. 


Exactly one of the isometries $g_{13}$ and $g_{23}$ is orientation-reversing, and it's straightforward to see that it must be a glide reflection: if $g_{ij}$ were a reflection, $\mathcal I_{ij}$ would coincide with the line segment $\overline{A_{ij}M}$, and the length of line segments on the boundary of each tile couldn't be the same.


\begin{proposition}\label{ReversingDisjoint}
$\mathcal S_3$ has an intersection of positive arc length with at least one of the curves $g_{13}(\mathcal S_1)$ and $g_{23}(\mathcal S_2)$.
\end{proposition}

\begin{proof}
Suppose for contradiction that $\mathcal S_3$ is disjoint from both curves, hence $g_{31}(\mathcal S_3)$ is disjoint (apart from possibly its endpoints) from $\mathcal S_1$.
As $\mathcal S_3\subseteq \partial\conv\mathcal D_3$, and no part of $\mathcal I_{12}$ is on $\partial\conv \mathcal D_1$, we clearly have $g_{31}(\mathcal S_3)\not\subseteq \mathcal I_{12}$. If $g_{13}$ is orientation-preserving, this yields that $g_{31}(A_{13})\in\relint \mathcal I_{13}$.

\begin{figure}[ht]
{\centering
\includegraphics[width=0.6\textwidth]{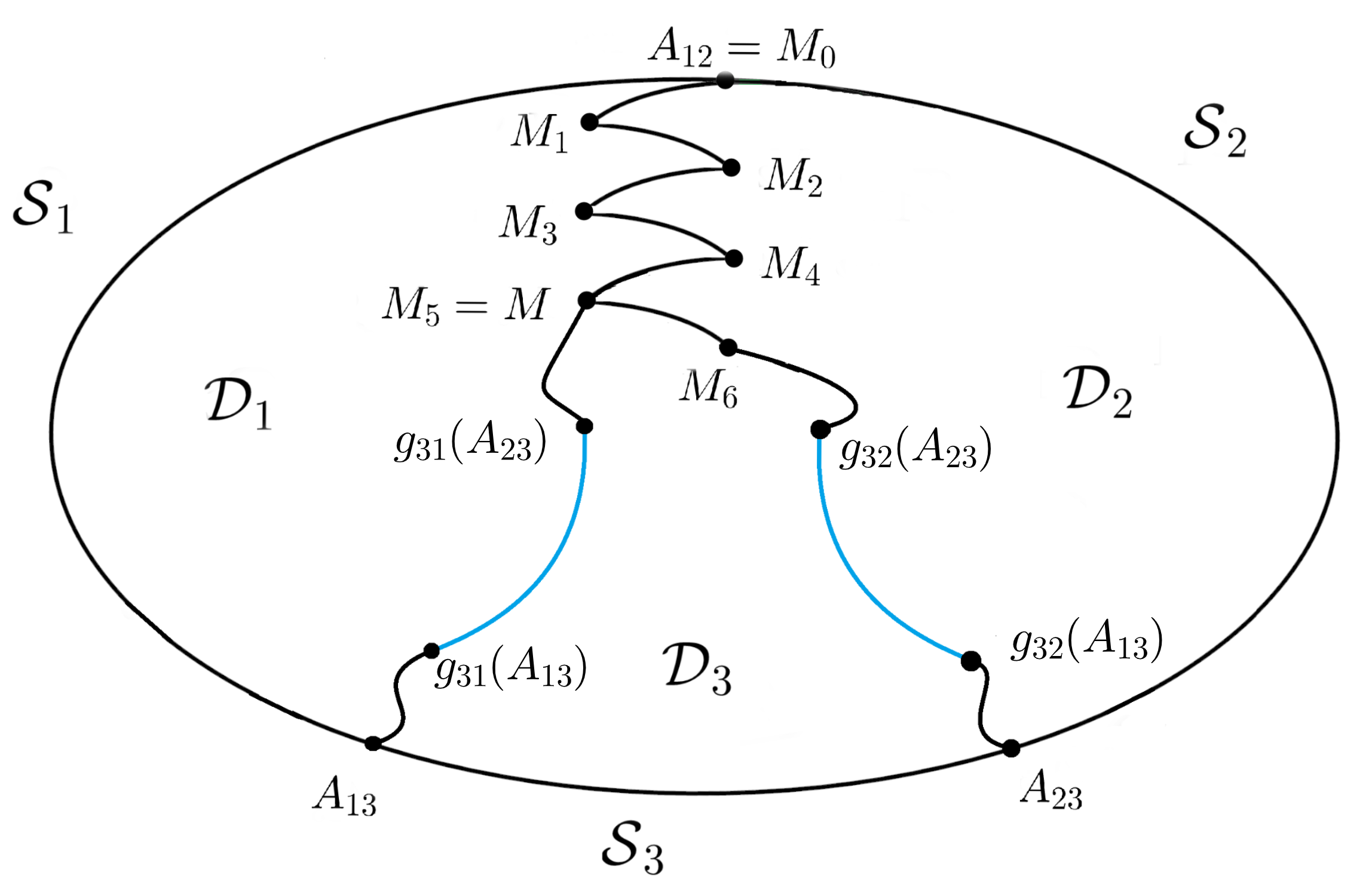}
\caption{\label{ReversingDisjointImg}}}
\end{figure}

If $g_{31}(A_{13})\neq A_{13}$, we obtain that the curve $\widehat{A_{13} g_{31}(A_{13})}$ is invariant under $g_{31}$, hence it's the point reflection about the midpoint of $\overline{A_{13}g_{31}(A_{13})}$. 
Then there is a part of the curve $\mathcal S_1$ in the neighbourhood of $A_{13}$ whose image under $g_{13}^2$ is on $\mathcal S_3$, a contradiction, as $g_{13}^2$ is the identity.

If $g_{31}(A_{13})=A_{13}$, $g_{13}$ is a rotation about $A_{13}$. 
Then there is a part of the curve $\mathcal S_1$ in the neighbourhood of $A_{13}$ whose image under $g_{13}^2$ is on $\mathcal S_3$, which is a contradiction due to the convexity of the curves.

If $g_{23}$ is orientation-preserving, then $\widehat{A_{23} g_{32}(A_{23})}$ is invariant under $g_{32}$, and we can use a similar argument to get a contradiction.
\end{proof}

\vspace{5pt}

Assume first that $g_{23}$ is orientation-reversing, hence $g_{13}$ is orientation-preserving. 
As we've already discussed the cases where $g_{ij}$ is orientation-preserving and $\length{\mathcal S_j\cap g_{ij}(\mathcal S_i)}>0$, we only present the cases where the intersection above is at most one point. 
Thus we have that $\length{\mathcal S_1 \cap g_{31}(\mathcal S_3)}= \length{g_{12}(\mathcal S_1) \cap g_{32}(\mathcal S_3)}=0$, and as $\mathcal S_3$ has an intersection of positive arc length with either $g_{13}(\mathcal S_1)$ or $g_{23}(\mathcal S_2)$, we have that $\length{\mathcal S_2\cap g_{32}(\mathcal S_3)}>0$.
From this, $\mathcal S_2\not\subseteq g_{12}(\mathcal S_1)$, hence $g_{12}(A_{13})\in\relint\mathcal S_2$.
This yields that $A_{13}$ is a cusp on $\partial\mathcal D_3$, thus $g_{32}(A_{13})$ can't be an interior point of $\mathcal S_2$. Hence their intersection is at most one point, which contradicts our assumptions.

From now on, we assume that $g_{13}$ is orientation-reversing -- hence a glide reflection --, $\length{\mathcal S_1\cap g_{31}(\mathcal S_3)}>0$ and $\length{\mathcal S_2\cap g_{32}(\mathcal S_3)}=0$, from which we have that $\mathcal S_1\not\subseteq g_{31}(\mathcal S_3)$.

\begin{proposition}
$g_{31}(A_{13})\in\relint\mathcal S_1$.
\end{proposition}

\begin{proof}
If $g_{31}(A_{23})\notin\relint \mathcal S_1$, $\mathcal S_1\not\subseteq g_{31}(\mathcal S_3)$ directly yields $g_{31}(A_{13})\in\relint\mathcal S_1$.
If $g_{31}(A_{23})\in\relint\mathcal S_1$, $A_{23}$ is differentiable on $\partial\mathcal D_3$, hence a cusp on $\partial\mathcal D_2$, thus $g_{21}(A_{23})\notin\relint\mathcal S_1$. 
If $g_{21}(A_{23})\neq A_{13}$, then $A_{13}$ is a cusp on $\partial\mathcal D_3$, thus $g_{31}(A_{13})\notin\mathcal S_1$. This yields that both $g_{31}(\mathcal S_3)$ and $g_{21}(\mathcal S_2)$ contain a part of $\mathcal S_1$ of positive length in the neighbourhood of $A_{13}$, which contradicts our assumptions.
If $g_{21}(A_{23})=A_{13}$, $A_{13}$ is a cusp on $\partial\mathcal D_1$ and differentiable on $\partial\mathcal D_3$, hence we have $g_{31}(A_{13})\in\relint\mathcal S_1$. \end{proof}

\vspace{5pt}

\begin{figure}[ht]
{\centering
\includegraphics[width=0.6\textwidth]{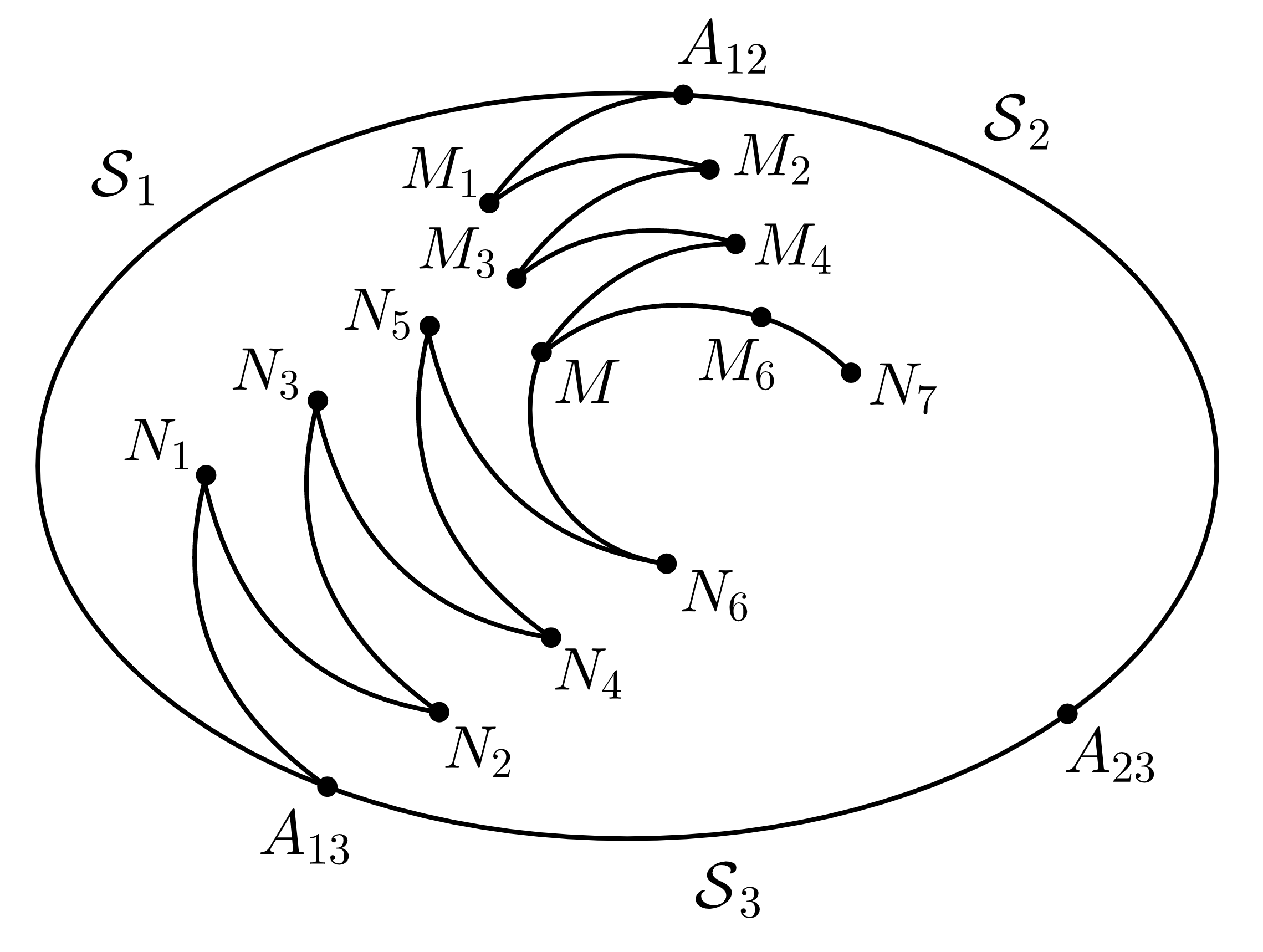}
\caption{\label{ReversingDoubleZigzag} $g_{12}$ and $g_{13}$ are glide reflections.}}
\end{figure}

\noindent Now, similarly to the construction of $\mathcal I_{12}$, we have $N_0=A_{13}$, $N_1=g_{13}(A_{13})$, and $N_{j+1}=g_{13}(N_j)$ if $N_j\in\mathcal I_{13}$, see Figure~\ref{ReversingDoubleZigzag}. 

If $M\in\relint{\widehat{M_{i-1}M_{i}}}$ and $M \in \relint{\widehat{N_{j-1}N_{j}}}$, the non-differentiable point closest to $M$ on $\mathcal I_{23}$ is $M_{i}=N_{j}$. This yields that $\widehat{M M_i} = \widehat{M N_j}$, where the former is the image of $\widehat{g_{21}^i(M) A_{12}}\subseteq \mathcal S_1$ under $g_{21}^i$, and the latter is the image of $\widehat{g_{31}^j(M) A_{13}}\subseteq \mathcal S_1$ under $g_{31}^j$.
Considering the convexity of the curves, we get that the parity of $i$ and $j$ must be different. Assume without loss of generality that $i$ is even and $j$ is odd. This yields that a translated image of $\widehat{g_{21}^i(M) A_{12}}$ coincides with the image of $\widehat{g_{31}^j(M) A_{13}}$ under a glide reflection. 
From this we have that the outer normal vector to $\partial K$ at $A_{12}$ coincides with the outer normal vector at $g_{31}^j(M)$. Given the strict convexity of $K$, this yields that $A_{12}=g_{31}^j(M)$, which implies that $g_{13}(\mathcal S_1)\cap \mathcal S_3=A_{13}$, a contradiction.

If $M=M_{i-1}=N_{j-1}$, $M$ is a cusp on all three tiles, which yields that $M_i$ and $N_j$ are both cusps as well, hence $M_i=N_j$, and we use a similar argument as in the previous case.

If $M=M_{i-1}$ and $M\in\relint \widehat{N_{j-1}N_j}$, we have $\widehat{M_{i-1}M_i} \subsetneq \widehat{N_{j-1}N_j}$. This yields that $g_{12}( \widehat{M_{i-2} N_{j-1}}) = \widehat{M N_j}$. This is a contradiction, as if $N_{j-1}$ is a positive cusp on $\partial\mathcal D_3$, $N_j$ is negative, and vice versa, thus they aren't identical on $\partial\mathcal D_1$ and $\partial\mathcal D_2$.

\noindent {\bf Subcase 2.b} $\; g_{12}$ is a reflection.

In this case, $g_{12}$ is a reflection about the line $A_{12}M$, and $\mathcal I_{12}$ coincides with the line segment $\overline{A_{12}M}$.
Exactly one of the isometries $g_{13}$ and $g_{23}$ is orientation-reversing, assume without loss of generality that it is the former. As seen in the previous part of the proof, this isn't a  glide reflection, thus it is a reflection: more precisely, the reflection about the line $A_{13}M$. 
This yields that $\mathcal I_{13}$ coincides with the line segment $\overline{A_{13}M}$. 
Thus we obtain that each tile consists of a convex differentiable curve congruent to $\mathcal S_1$, and two line segments. 
Hence $\mathcal I_{23}$ can't contain a stricly concave curve on neither $\partial\mathcal D_2$ nor $\partial\mathcal D_3$, thus it is also a line segment. 
It follows that $g_{13}(\mathcal S_1)=g_{23}(\mathcal S_2)=\mathcal S_3$; that the line segments $\mathcal I_{ij}$ $(1\leq i<j\leq 3)$ are of equal length, and as the angle between any two of them is the same, that angle is $120^\circ$.

\begin{remark}\label{NonNormalTopII}
If $\length{\mathcal S_3}=0$, the tiles have two points of intersection: one on the boundary of $K$, and one in the interior. We denote the point in the interior by $M$, and use the same further notation as in the case $\length{\mathcal S_3}>0$. Here, we have that $A_{13}=A_{23}$, and we will refer to this point as $A$ for simplicity.
In Subcases 1.a-1.d, the calculations hold in the same way as above. In Subcase 1.e, $g_{12}^2$ has two fixed points: $A_{12}$ and $A$. As $g_{12}$ is orientation-preserving, this yields that $g_{12}$ is the point reflection about the midpoint of $\overline{A_{12}A}$. From this we have $\mathcal J_1=\mathcal J_3$, a contradiction.

Now, consider Subcase 2.a.
If $g_{31}(A)\in\relint\mathcal S_1$, $A$ is a cusp on both $\partial\mathcal D_1$ and $\partial\mathcal D_2$, which yields that $A$ is a fixed point of $g_{12}$, a contradiction.
If $g_{31}(A)\in\mathcal I_{13}\setminus \{A\}$, we obtain a contradiction similarly to Proposition~\ref{ReversingDisjoint}.
If $g_{31}(A)=A$, $g_{13}$ is a rotation around $A$, hence the angle of $A$ on $\partial\mathcal D_1$ is equal to the angle of the rotation. If $g_{13}$ is a point reflection, we get a contradiction similarly to Proposition~\ref{ReversingDisjoint}; otherwise we have that $A$ is an ordinary non-differentiable point on $\partial\mathcal D_1$, thus $g_{12}(A)\notin \relint \mathcal S_2$. $g_{12}(A)=A$ is a contradiction as $g_{12}$ has no fixed points; $g_{12}(A)\notin\mathcal S_2$ yields that $A$ is a cusp on $\partial\mathcal D_1$, which contradicts the fact that it is ordinary.

\begin{figure}[ht]
{\centering
\includegraphics[width=\textwidth]{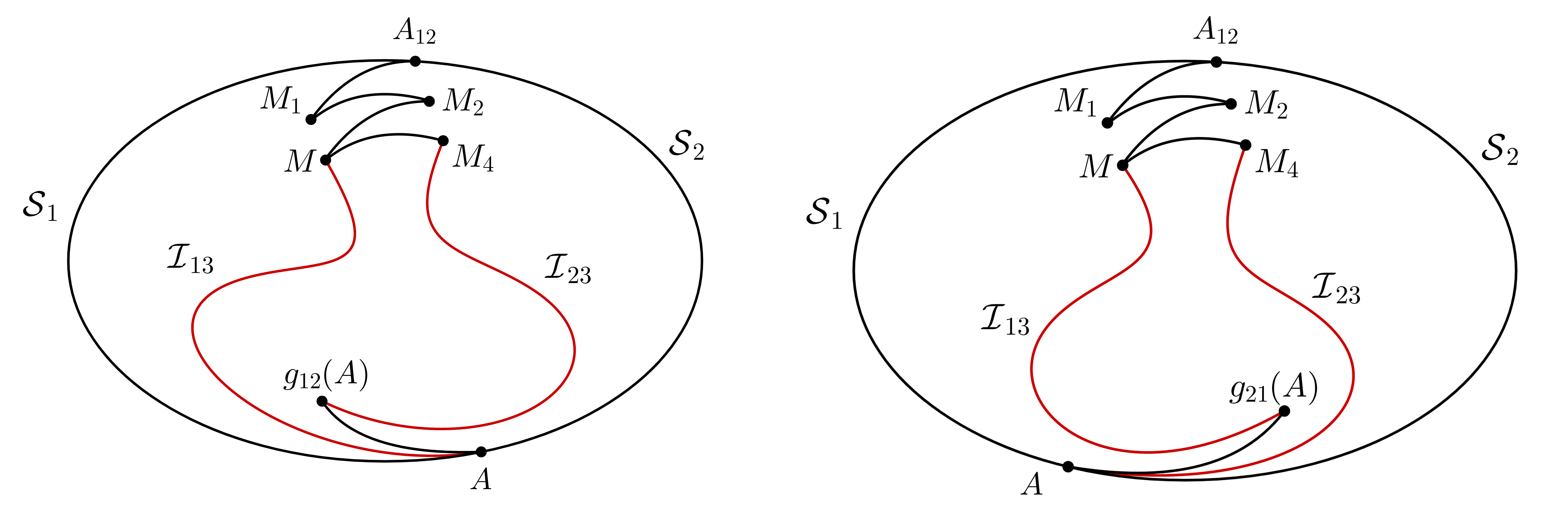}
\caption{\label{S3.0}}
}
\end{figure}

Lastly, consider the case $g_{31}(A)\in \mathcal I_{12}$. As $A\in\partial \conv \mathcal D_3$, and the only points on $\mathcal I_{23}$ that can be on $\partial \conv \mathcal D_1$ are $M_{2i}$, we have that $g_{31}(A)=M_{2i}$ for an integer $i$.

It is easy to see that $g_{12}(A)\in\mathcal I_{12}$ or $g_{21}(A)\in\mathcal I_{12}$ yields a contradiction due to the length of the boundaries.
If $g_{12}(A)\in\relint \mathcal  I_{23}$, the convexity of the curves yields that $g_{13}$ is orientation-reversing, thus $g_{23}$ is orientation-preserving. It is easy to see that as $g_{12}(\mathcal I_{13})\subsetneq \mathcal I_{23}$, see Figure~\ref{S3.0}. Hence we have $\length{\mathcal I_{13}}<\length{\mathcal I_{23}}$. This yields that $\mathcal I_{23}\cap g_{23}(\mathcal I_{23})$ is a non-degenerate curve, and its image under $g_{23}$ is itself. Then $g_{23}$ is a point-reflection, a contradiction. 

If $g_{12}(A)\in\relint\mathcal S_2$, thus $g_{21}(A)\in\relint\mathcal I_{13}$, the convexity of the curves yields that $g_{13}$ is orientation-preserving. Similarly to our previous arguments, if the image of $\mathcal S_1 \cup g_{21}(\mathcal S_2)$ or $\widehat{g_{21}(A)M}$ under $g_{13}$ has a non-degenerate intersection with $\widehat{g_{21}(A)M}$, $g_{13}$ is a point reflection, and we have a contradiction. This yields that only $g_{13}(\mathcal I_{12})$ can be on $\widehat{g_{21}(A)M}$. Comparing the length of the boundaries, it is easy to see that $\length{\mathcal I_{12}} < \length{ \widehat{ g_{21}(A) M}}$, thus this is a contradiction as well.
\end{remark}

\section {Some partial results on non-normal tilings}\label{NonNormal}

In this section we prove a partial result for non-normal tilings. We assume that $K$ is centrally symmetric, and it is unclear if this case can occur or not, see the Concluding remarks.

\begin{theorem}\label{thm:partial}
Let $K$ be a centrally symmetric, strictly convex disc with a continuously differentiable boundary, and $\{ \mathcal D_1, \mathcal D_2, \mathcal D_3 \}$ be a non-normal tiling of $K$. Assume that the boundary of the tile $\mathcal D_i$ is piecewise differentiable for $i=1,2,3$, and that $\mathcal D_1\cap \mathcal D_2$ is not connected. Then $\mathcal S_3=\emptyset$, the isometry $g_{12}$ is a rotation, while $g_{13}$ and $g_{23}$ are glide reflections.
\end{theorem}

\begin{proof}

We denote the center of symmetry of $K$ by $O$, and the point reflection about $O$ by $\tau_o$. 
By Remarks \ref{NonNormalTopI} and \ref{NonNormalTopII}, we obtain a contradiction for the non-normal tilings of topological cases I and II (see Figure~{\ref{TopCases}}).
In case III it readily follows that $\mathcal S_3$ is empty. By Lemma~\ref{intersection}, we have that $\length{g_{12}(\mathcal S_1)\cap \mathcal S_2}>0$. 
Let the two points of intersection of the curves $\mathcal S_1$ and $\mathcal S_2$ be $A$ and $B$, and the two intersection point of all three tiles be $M$ and $N$. We distinguish cases according to the type of $g_{12}$.

\noindent{\bf Case 1} $\;g_{12}$ is a reflection.

Let the axis of reflection be $t$. Note that $t$ cannot contain an interior point of $\mathcal D_1$ or $\mathcal D_2$, thus $\mathcal D_1 \cap \mathcal D_2$ consists of two line segments contained in $t$, and $t$ coincides with the line $AB$.
The convex hulls of $\mathcal D_1$ and $\mathcal D_2$ are, respectively, the topological discs enclosed by $\mathcal S_1$ and $\overline{AB}$, and $\mathcal S_2$ and $\overline{AB}$.
Now, $\mathcal D_3$ is symmetric about the line $AB$, thus $\mathcal D_1$ is axially symmetric about some line $l$ as well. As the convex hull of $\mathcal D_1$ is symmetric about $l$, $l$ must be the bisector of the line segment $\overline{AB}$. This yields that $\mathcal D_3$ has two lines of symmetry -- $AB$ and the bisector of $\overline{AB}$--, while $\mathcal D_1$ only has one, a contradiction.

\noindent {\bf Case 2} $\;g_{12}$ is a translation. 

For a convex differentiable arc $\mathcal C$ we denote by $\tc \mathcal C$ the usual total curvature of $\mathcal C$, that is $\tc \mathcal C$ is the positive turning angle of the tangent vectors at the endpoints of $\mathcal C$. We may assume without loss of generality that $\tc (\mathcal S_1) \leq \tc(\mathcal S_2)$, thus $\tc (\mathcal S_2)\geq \pi$. This yields that there exists a point $A'\in \mathcal S_2$ that satisfies the property $\tc(AA')=\pi$. 
The image of $\overline{AA'}$ under $g_{21}$ is a line segment parallel to $\overline{AA'}$ that is contained in the convex hull of $\mathcal D_1$, thus contained in $K$. 
As the tangents to $\partial K$ at $A$ and $A'$ are parallel to each other, this can only occur if $g_{21}(\overline{AA'}) = \overline{AA'}$, which yields that $g_{21}$ is the identity, a contradiction.

\noindent {\bf Case 3} $\;g_{12}$ is a glide reflection.

In this case, assume that $\length{\mathcal S_1}\leq \length{\mathcal S_2}$. We distinguish two further subcases.

\noindent {\bf Subcase 3.a} $\;g_{12}(\mathcal S_1)\subseteq\mathcal S_2$.

Here we have $\tc (\mathcal S_1) \leq \tc(\mathcal S_2)$, and $\tc (\mathcal S_2)\geq \pi$.

As the total curvature of $\widehat{g_{12}(A)g_{12}(B)} \subseteq \mathcal S_2$ is equal to $\tc(\mathcal S_1)$, we  have that on $\mathcal S_2$, $\tc(\widehat{Bg_{12}(A)})+ \tc(\widehat{g_{12}(B)A}) =\tc(\mathcal S_1)+ \tc(\mathcal S_2) = 2\pi$, thus either $\tc(\widehat{Bg_{12}(A)})\geq \pi$ or $\tc(\widehat{g_{12}(B)A})\geq \pi$ holds. Assume that $\tc(\widehat{g_{12}(B)A})\geq \pi$.
Then there exists a point $B'\in \widehat{Ag_{12}(A)}$ such that $\tc(\widehat{g_{12}(B)B'})=\pi$ holds.
Then, $g_{21}(g_{12}(B)) = B\in\mathcal S_2$, and $g_{21}(B')\in \partial\conv \mathcal D_1\setminus \mathcal S_1$, which is a subset of $\conv \mathcal D_2$. Hence we have that $g_{21}( \widehat{g_{12}(B) B')} \subset \conv \mathcal D_2$, and $g_{21}^2 (\widehat{g_{12}(B) B')} \subset \conv \mathcal D_1\subset K$. As $g_{21}^2$ is a translation and the tangent lines to $\partial K$ at $g_{12}(B)$ and $B'$ are parallel to each other, this yields that $g_{21}$ has a translation vector of $\mathbf 0$, which contradicts our assumptions.

\noindent {\bf Subcase 3.b}  $\;g_{12}(\mathcal S_1)\not\subseteq \mathcal S_2$.

Let $t$ denote the axis of reflection of $g_{12}$. 
Let $X$ and $X'$ denote the points on $\partial K$ where the tangents are parallel to $t$, and assume that $X\in \mathcal S_1$. 
Then $g_{12}(X)=X'$, hence $X'\in\mathcal S_2$. As $g_{12}(X)\neq g_{21}(X)$, we have $g_{21}(X)\notin K$, so $X\notin \mathcal S_2$, and similarly, $X'\notin \mathcal S_1$. 
The midpoint of the line segment $\overline{XX'}$ is $O$, thus $O\in t$.
Consider the lines parallel to $t$ passing through $A,\,B,\, g_{21}(A)$ and $g_{12}(B)$. These lines determine four, pairwise disjoint curves on $\partial K$ that we denote by $\mathcal C_i$ $(i=1,2,3,4)$, see Figure~\ref{Case33b}. Note that the light yellow strips on Figure~\ref{Case33b} are symmetric about $t$.

\begin{figure}[h]
{\centering
\includegraphics[width=0.75\textwidth]{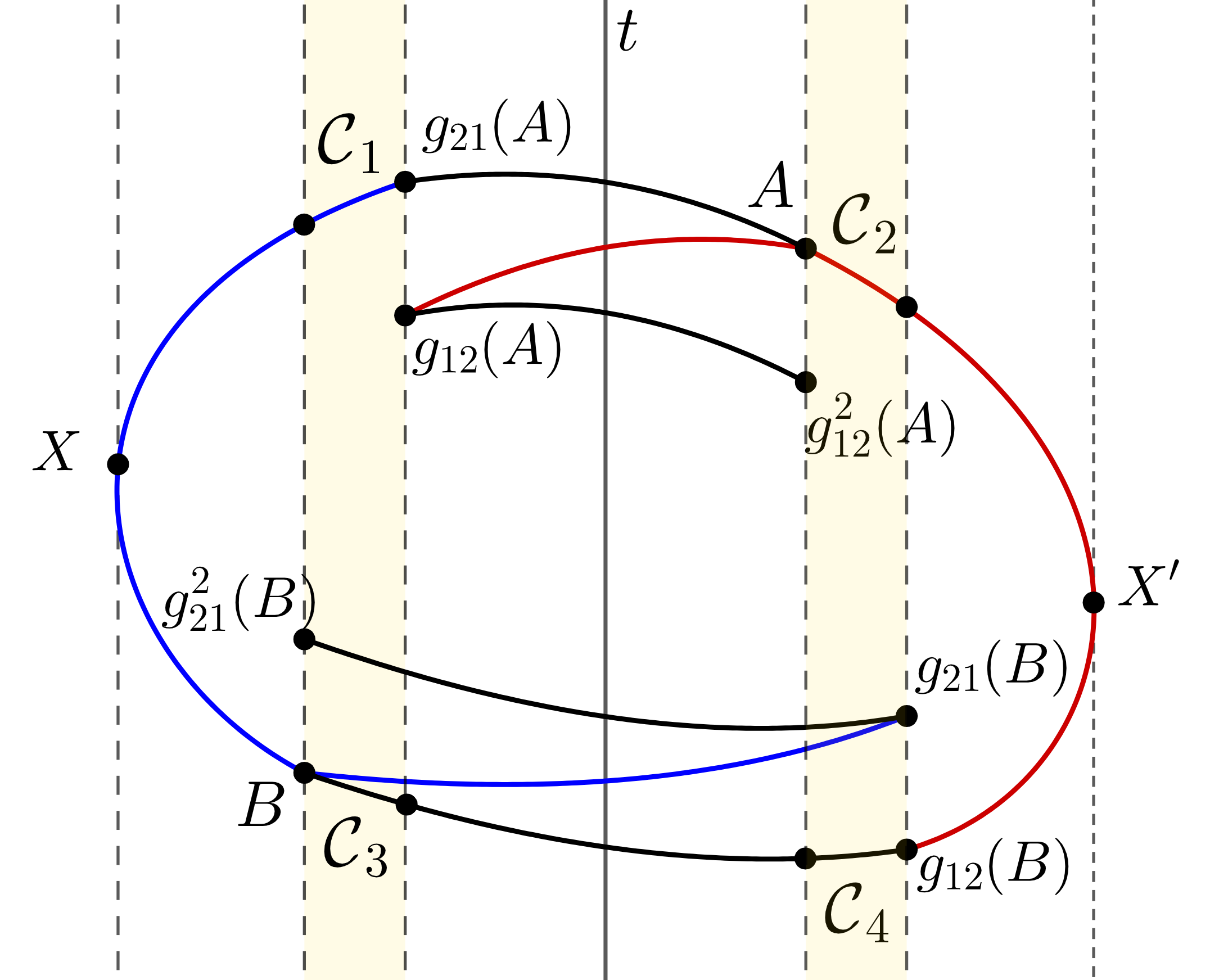}
\caption{\label{Case33b}}
}
\end{figure}

Then we have that $\tau_o( g_{12}( \tau_o( \mathcal C_4))) = \mathcal C_3$, and as $O\in t$, $\tau_o\circ g_{12}\circ \tau_o=g_{21}$. Thus $g_{21}(\mathcal C_4)=\mathcal C_3$, and since $\mathcal C_3 \subseteq \mathcal S_2$, this only holds if $\mathcal C_3$ --and subsequently, every $\mathcal C_i$-- consists of a single point. 
This yields that $A$ and $B$ are antipodal points (that is, $A$ and $B$ are centrally symmetric with respect to $O$), $\tc(\mathcal S_1) = \tc(\mathcal S_2) = \pi$, and $\widehat{Ag_{12}(A)} \cong \widehat{Bg_{21}(B)}$, hence the roles of $\mathcal D_1$ and $\mathcal D_2$ are interchangeable. 
Thus we may assume without loss of generality that $g_{13}$ is an orientation-preserving isometry.
We consider three cases based on the position of $g_{13}(\mathcal D_1\cap \mathcal D_3)$.

First, assume $g_{13}(\mathcal D_1\cap\mathcal D_3)=\mathcal D_1\cap \mathcal D_3$. In this case, we have $g_{13}(M)=N$ and $g_{13}(N)=M$. This yields that $g_{13}$ must be the reflection through the midpoint of $\overline{MN}$. 
As the tangent lines to $\partial K$ at $A$ and $B$ are parallel to each other, $g_{13}(\overline{AB})$ can't be in the interior of $K$, which contradicts $\mathcal D_3 \subset \inter K$.

Secondly, if we have $g_{13}(\mathcal D_1\cap\mathcal D_3)=\mathcal D_2\cap\mathcal D_3$, both $M$ and $N$ are fixed points of $g_{13}$, which, in the case on an orientation-preserving isometry, yields that $g_{13}$ is the identity, a contradiction.

Lastly, consider the case where $g_{13}(M)\in\relint (\mathcal D_1\cap\mathcal D_3)$. 
Then the image of curve $\widehat{Mg_{13}(M)}$ under $g_{13}$ is itself, hence $g_{13}$ is the reflection through the midpoint of $\overline{Mg_{13}(M)}$. 
Considering a point $Y\in\mathcal D_1\cap\mathcal D_3$ in a suitably small neighbourhood of $g_{13}(M)$ in a way that $Y\notin g_{13}(\mathcal D_1\cap\mathcal D_3)$, we have that $g_{13}(Y)\in\relint(\mathcal D_2\cap\mathcal D_3)$ and $g_{31}(Y)\in\partial\mathcal D_1\setminus \mathcal D_3$,
but $g_{13}(Y)=g_{31}(Y)$, a contradiction.

\noindent {\bf Case 4} $\;g_{12}$ is a rotation.

First, consider the case where $g_{12}(\mathcal S_1)=\mathcal S_2$. In this case, we have that $A$ and $B$ are antipodal, and $g_{12}$ is the point reflection about $O$. This yields that $\mathcal D_3$ is symmetric about $O$, thus $\mathcal D_1$ is also symmetric about some point $P$. The reflected image of $\overline{AB}$ about $P$ is a line segment parallel to $\overline{AB}$, in $\mathcal D_1$. As $A$ and $B$ are antipodal, this can only be the case if $P=O$, which is clearly a contradiction.

From now on, we have $g_{12}(\mathcal S_1)\neq \mathcal S_2$. Without loss of generality, we may assume that $\length{\mathcal S_1}\leq \length{\mathcal S_2}$, and $g_{21}(B)\notin \mathcal S_1$. This yields that $A$ is a cusp on $\partial\mathcal D_2$, thus $g_{21}(A)\notin \relint \mathcal S_1$. 
This implies that $\mathcal S_1 \subsetneq g_{21}(\mathcal S_2)$,  hence $\length{\mathcal S_2}$ is strictly greater than $\length{\mathcal S_1}$, and $\tc (\mathcal S_1) < \tc (\mathcal S_2)$

We will present the case $g_{21}(A), g_{21}(B)\in \mathcal D_1\cap \mathcal D_2$. The case where $g_{21}(A)$ or $g_{21}(B)$ is an interior point of $\mathcal D_1\cap \mathcal D_3$ can be examined similarly.

\noindent {\bf Subcase 4.a} $\;g_{13}$ is orientation-preserving.

As $\partial\mathcal D_1$ and $\partial\mathcal D_2$ are of equal length, we have $\length{\mathcal D_1\cap \mathcal D_3} > \length{\mathcal D_2\cap \mathcal D_3}$, thus the intersection of $\mathcal D_1\cap \mathcal D_3$ with its image under $g_{13}$ is a non-degenerate curve.
Similarly to Subcase 3.b, if $g_{13}(M)$ or $g_{13}(N)$ is an interior point of $\mathcal D_1\cap \mathcal D_3$, we get a contradiction. Hence the image of $\mathcal D_1\cap \mathcal D_3$ under $g_{13}$ must be itself, and $g_{13}(M)=N$, $g_{13}(N)=M$. In the case of an orientation-preserving isometry, this yields that $g_{13}$ is the reflection through the midpoint of the line segment $\overline{MN}$.
From this we have that $g_{13}(g_{21}(\mathcal S_2))=g_{23}(\mathcal S_2)\subseteq \mathcal D_2\cap \mathcal D_3$.
We will examine the image of the curve $g_{23}(\mathcal S_2)$ under $g_{23}$ and $g_{21}$. See Figure~\ref{Case34a}.

\begin{figure}[h]
{\centering
\includegraphics[width=0.7\textwidth]{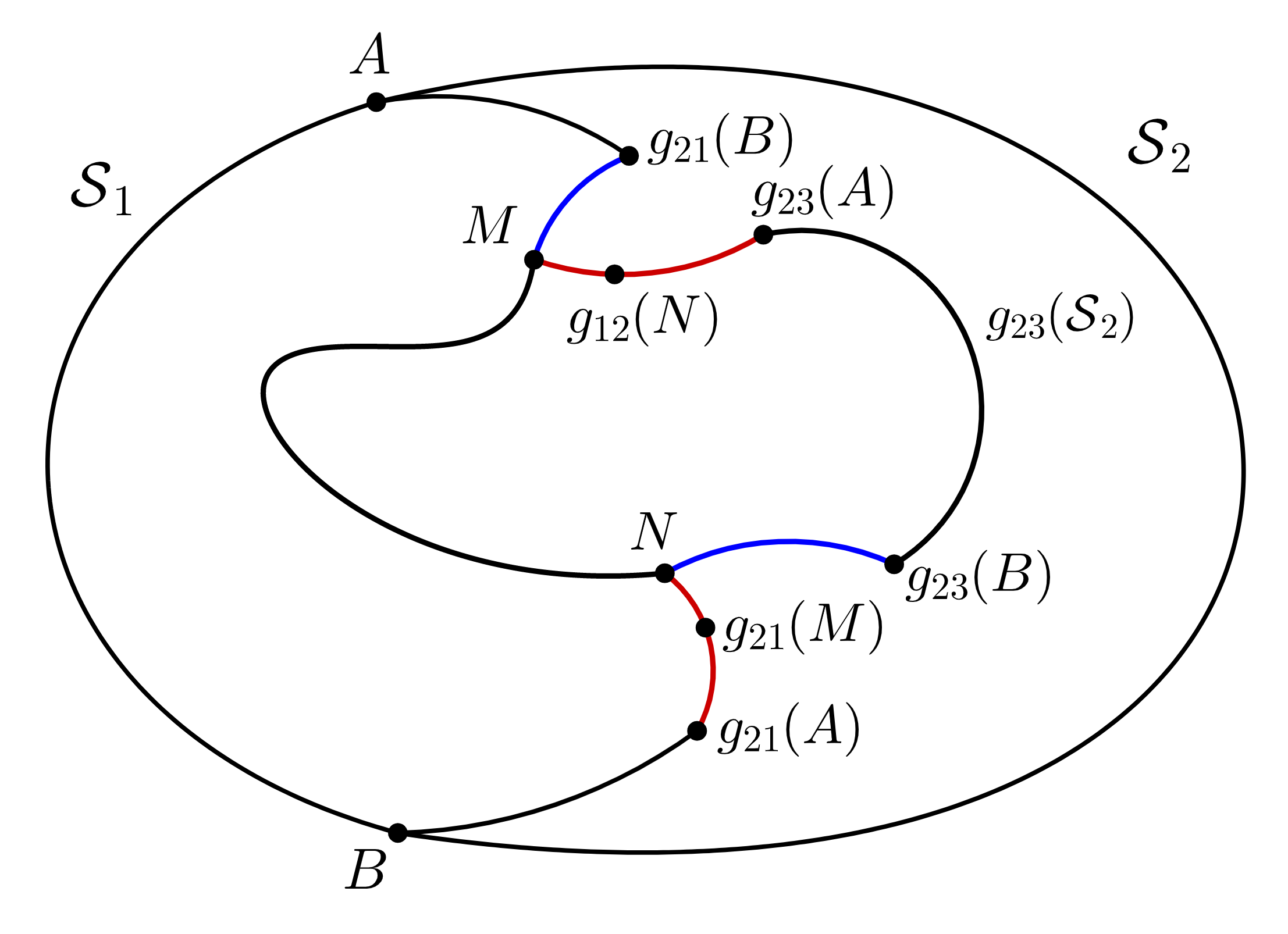}
\caption{\label{Case34a}}
}
\end{figure}

$g_{23}(\mathcal S_2)$ begins at a distance of $\length{\widehat{AM}}+ \length{ \widehat{Mg_{23}(A)}}$ from $A$ of $\partial\mathcal D_2$, thus $g_{23}^2(\mathcal S_2)$ is on $\mathcal D_1\cap \mathcal D_3$, beginning at a distance of $\length{\widehat{AM}}$ from $M$.
On the other hand, $g_{23}(\mathcal S_2)$ begins at a distance of $\length{\widehat{BN}} + \length{ \widehat{Ng_{23}(B)}}$ from $B$, and $\length{Ng_{23}(B)}=\length{g_{21}(B)M}$, thus $g_{21}(g_{23}(\mathcal S_2))$ is on $\mathcal D_1\cap \mathcal D_3$, beginning at a distance of $\length{\widehat{BN}}$ from $M$.
As the convexity of $g_{23}^2(\mathcal S_2)$ and $g_{21}(g_{23}(\mathcal S_2))$ is different, the difference between the lengths is at least $\length{\mathcal S_2}$. 
Without loss of generality, we obtain that $\length{\widehat{BN}} \geq \length{\widehat{AM}} + \length{\mathcal S_2} > \length{\widehat{AM}} + \length{\widehat{Bg_{21}(A)}}$. From this, $\length{\widehat{g_{21}(A)N}} > \length{\widehat{AM}}$. This yields that $g_{21}(M)\in \relint \widehat{g_{21}(A) N}$, hence $g_{12}(N)\in \relint \widehat{M g_{23}(A)}$, and we see that $g_{12}(N) = g_{13} (g_{21}(M))$.
Thus we have that the image of $\widehat{g_{21}(M)N}$ under $g_{12}$ is $\widehat{M g_{12}(N)}$, and under $g_{13}$ it's $\widehat{g_{12}(N) M}$. As $g_{13}$ is a point reflection, the second property yields that the line segment $\overline{g_{21}(M)N}$ is parallel to $\overline{g_{12}(N)M}$, but as $g_{12}$ is not a point reflection, this is a contradiction.

\noindent {\bf Subcase 4.b} $\; g_{13}$ is orientation-reversing.

If $g_{13}$ is a reflection, $\mathcal D_1\cap \mathcal D_3$ is a line segment, hence $g_{23}(\mathcal S_2)\subseteq \mathcal D_2\cap \mathcal D_3$. This yields that $g_{23}^2(\mathcal S_2)\subset \mathcal D_3\subset K$, which is a contradiction similarly to case 1.

If $g_{23}$ is a reflection, we clearly have that the axis of reflection has no common points with $\mathcal S_2$. As $\tc (\mathcal S_2) > \pi$, it is easy to see that $g_{23}(\mathcal S_2)\not \subset K$.

\end{proof}

\section{Concluding remarks}

In Theorem~{\ref{thm:partial}}, we state a positive result, however, we conjecture that topological case III cannot occur, even without the assumption that $K$ is centrally symmetric. We formulate the following conjecture.

\begin{conj}
Theorem~\ref{three} holds true for non-normal monohedral tilings as well.
\end{conj}

Differentiability was a main tool in the proofs. However it is unclear that apart from convexity which assumptions are essential to obtain similar results, see Figure~\ref{Nondiff}
for some examples. We challenge the reader to relax on the conditions of Theorem~\ref{three}.

\section{Acknowledgements}
K. Nagy was partially supported by the ÚNKP-21-1 New National Excellence Program of
the Ministry for Innovation and Technology from the source of the
National Research, Development and Innovation Fund.

V. V\'\i gh was partially supported by Hungarian NKFIH grant FK135392. 

This research was supported by grant NKFIH-1279-2/2020 of the Ministry for Innovation and Technology, Hungary.

\end{document}